\documentclass{amsart}
\usepackage{amsmath,amscd,amssymb,stmaryrd}
\usepackage{amsfonts}
\usepackage[all]{xy}
\usepackage{mathtools}
\usepackage{enumerate}
\usepackage[retainorgcmds]{IEEEtrantools}
\usepackage{color}
\numberwithin{equation}{subsection}
\theoremstyle{plain}
\newtheorem{thm}[subsection]{Theorem}

\newtheorem{prop}[subsection]{Proposition}
\newtheorem{lemma}[subsection]{Lemma}
\newtheorem{cor}[subsection]{Corollary}

\theoremstyle{definition}
\newtheorem{defn}[subsection]{Definition}
\newtheorem{examples}[subsection]{Examples}
\newtheorem{notn}[subsection]{Notation}

\newtheorem{ackn}[subsection]{Acknowledgement}
\theoremstyle{remark}
\newtheorem{rem}[subsection]{Remark}

\begin{document}
\title{The strong Massey vanishing conjecture for fields with virtual cohomological dimension at most $1$}
\author{Ambrus P\'al and Endre Szab\'o}
\date{March 25, 2022.}
\address{Department of Mathematics, 180 Queen's Gate, Imperial College, London, SW7 2AZ, United Kingdom}
\email{a.pal@imperial.ac.uk}
\address{Alfr\'ed R\'enyi Institute of Mathematics, Budapest, Re\'altanoda u.~13-15, H-1053, Hungary}
\email{szabo.endre@renyi.mta.hu}
\begin{abstract} We show that a strong vanishing conjecture for $n$-fold Massey products holds for fields of virtual cohomological dimension at most $1$ using a theorem of Haran. We also prove the same for PpC fields, using results of Haran--Jarden. Finally we construct a pro-$2$ group which satisfies the weak Massey vanishing property for every $n\geq3$, but does not satisfy the strong Massey vanishing property for $n=4$.  
\end{abstract}
\maketitle
\pagestyle{myheadings}
\markboth{Ambrus P\'al and Endre Szab\'o}{Massey vanishing for virtual cohomological dimension $1$}

\section{Introduction}

\begin{defn}\label{11.1} Let $\mathcal C^*$ be a differential graded associative algebra with product $\cup$, differential $\delta:\mathcal C^*\to\mathcal C^{*+1}$, and cohomology $H^*=\textrm{Ker}(\delta)/\textrm{Im}(\delta)$. Choose an integer $n\geq2$ and let $a_1,a_2,\ldots,a_n$ be a set of cohomology classes in $H^1$. A defining system for the $n$-fold Massey product of $a_1,a_2,\ldots,a_n$ is a set $a_{ij}$ of elements of $\mathcal C^1$ for $1\leq i<j\leq n+1$ and $(i,j)\neq(1,n+1)$ such that
$$\delta(a_{ij})=\sum_{k=i+1}^{j-1}a_{ik}\cup a_{kj}$$
and $a_1,a_2,\ldots,a_n$ is represented by $a_{12},a_{23},\ldots,a_{n,n+1}$. We say that the $n$-fold Massey product of $a_1,a_2,\ldots,a_n$ is defined if there exists a defining system. The $n$-fold Massey product $\langle a_1,a_2,\ldots,a_n\rangle_{a_{ij}}$ of $a_1,a_2,\ldots,a_n$ with respect to the defining system $a_{ij}$ is the cohomology class of
$$\sum_{k=2}^na_{1k}\cup a_{k,n+1}$$
in $H^2$. Let $\langle a_1,a_2,\ldots,a_n\rangle$ denote the subset of $H^2$ consisting of the $n$-fold Massey products of $a_1,a_2,\ldots,a_n$ with respect to all defining systems. We say that the $n$-fold Massey product of $a_1,a_2,\ldots,a_n$ vanishes if $\langle a_1,a_2,\ldots,a_n\rangle$ contains zero.
\end{defn}
\begin{defn}\label{1.2} Let $p$ be a prime number, let $G$ be a profinite group, let $\mathcal C^*$ be the differential graded algebra of $\mathbb Z/p$-cochains of $G$ in continuous group cohomology. The cohomology of $\mathcal C^*$ is $H^*=H^*(G,\mathbb Z/p)$. We say that $G$ has the {\it strong Massey vanishing property for $n$ with respect to $p$}, where $n$ is an integer $\geq3$, if for every $a_1,a_2,\ldots,a_n\in H^1(G,\mathbb Z/p)$ such that $a_i\cup
a_{i+1}=0$ for every $1\leq i<n$, the $n$-fold Massey product of $a_1,a_2,\ldots,a_n$ vanishes. We say that $G$ has the {\it strong Massey vanishing property with respect to $p$} if it has the strong Massey vanishing property with respect to $p$ for every integer $n\geq3$. We say that $G$ has the {\it weak Massey vanishing property with respect to $p$} if for every integer $n\geq3$ and $a_1,a_2,\ldots,a_n\in H^1(G,\mathbb Z/p)$ such that $n$-fold Massey product of $a_1,a_2,\ldots,a_n$ is defined, the $n$-fold Massey product of $a_1,a_2,\ldots,a_n$ vanishes.
\end{defn}
\begin{defn}\label{5.2} For every field $K$ let $G(K)$ denote the absolute Galois group of $K$. Assume that the characteristic of $K$ is not $p$. We say that the {\it strong Massey vanishing conjecture with respect to $p$} holds for $K$ if $G(K)$ has the strong Massey vanishing property with respect to $p$. 
\end{defn}
\begin{rem} We call our conjecture strong because it is stronger in general than the Massey vanishing conjecture formulated by Min\'a\v c and T\^an (Conjecture 1.6 of \cite{MT2} on page 259) since we do not require that the $n$-fold Massey product of $a_1,a_2,\ldots,a_n$ is defined, unlike them. This is a strictly stronger requirement when $n>3$, see Theorem \ref{counter0} below. See also Remark \ref{compare} below, which explains the difference in terms of embedding problems. 
\end{rem}
There is quite a bit of beautiful work on this fascinating conjecture (see for example \cite{EM1}, \cite{GMTT}, \cite{HaW}, \cite{HW}, and \cite{MT2}), but it remains open in general. Our aim is to prove the strong form of this conjecture for two new classes of fields whose definition we recall next. 
\begin{defn} Recall that a field $K$ has virtual cohomological dimension $\leq 1$ if there is a finite separable extension $L/K$ with $\mathrm{cd}(L)\leq 1$ where $\mathrm{cd}$ denotes the cohomological dimension as defined in \cite{Se}. Since the only torsion elements in the absolute Galois group of $K$ are the involutions coming from the orderings of $K$, it is equivalent (by a theorem in \cite{Se2}) to require $\mathrm{cd}(L)\leq 1$ for any fixed finite separable extension $L$ of $K$ without orderings, for example for $L= K(\mathbf i)$, where $\mathbf i =\sqrt{-1}$. In particular, if $K$ itself cannot be ordered (which is equivalent to $-1$ being a sum of squares in $K$), this condition is equivalent to $\mathrm{cd}(K)\leq 1$.
\end{defn}
\begin{examples} Examples of fields $K$ which can be ordered with
$\mathrm{cd}(K(\mathbf i))\leq 1$ include real closed fields, function fields in one variable over any real closed ground field (by Tsen's Hauptsatz of \cite{Ts} on page 335), PRC (pseudo real closed) fields (for definition see page 450 of \cite{HJ}), the field of Laurent series in one variable over any real closed ground field (by Lang's Theorem 10 of \cite{Lg} on page 384), and the field $\mathbb Q^{ab}\cap\mathbb R$ which is the subfield of $\mathbb R$ generated by the numbers $\cos(\frac{2\pi}{n})$, where $n\in\mathbb N$ (see Corollary 6.2 of \cite{Ha0} on page 410).
\end{examples}
The first main result of this paper is the following
\begin{thm}\label{massey_strong} The strong Massey vanishing conjecture holds for fields $K$ with $\mathrm{cd}(K(\mathbf i))\leq 1$ with respect to every prime number.
\end{thm}
The proof is an easy application of earlier work of Haran and Dwyer. After we review the latter, we give a quick proof of our first main result in the next section. At the recommendation of the reviewer, we will use similar methods to prove the same claim for pseudo $p$-adically closed fields, whose definition we recall next. 
\begin{defn} A field $K$ is called pseudo $p$-adically closed (abbreviation: PpC) if every absolutely irreducible variety $V$ defined over $K$ has
a $K$-rational point, provided $V$ has a $L$-rational simple point for each $p$-adic closure $L$ of $K$. 
\end{defn}
\begin{thm}\label{main_p-adic} The strong Massey vanishing conjecture holds for $\mathrm{PpC}$ fields with respect to every prime number.
\end{thm}
The result follows from a result of Haran--Jarden (see Theorem \ref{hj2} below), and the validity of the strong Massey vanishing property for Demushkin groups, which is our Theorem \ref{demushkin_vanishing}, and does require some work. Our last main result is the following purely group-theoretical
\begin{thm}\label{counter0} There is a pro-$2$ group $G$ which satisfies weak Massey vanishing for $n\geq3$ (with respect to every prime number), but does not satisfy strong Massey vanishing for $n=4$ (with respect to $2$).
\end{thm}
The key idea of the proof of this theorem is the use of {\it Massey envelopes} which are infinite fibre products that can be associated to each group, satisfy weak Massey vanishing, and which are in some sense universal with respect to this property. 
\begin{ackn} The first author wishes to acknowledge the generous support of the Imperial College Mathematics Department's Platform Grant. The second author was supported by the National Research, Development and Innovation Office (NKFIH) Grants K120697, K115799. The project leading to this application has received funding from the European Research Council (ERC) under the European Union's Horizon 2020 research and innovation programme (grant agreement No 741420). We also wish to thank the referee for his useful comments and for suggesting us to prove Theorem \ref{main_p-adic}, too, and for his intriguing questions leading us to Theorem \ref{counter0}. 
 \end{ackn}

\section{Real embedding problems and Massey products}

\begin{defn}\label{2.1} An embedding problem for a profinite group $G$ is the left hand side diagram:
$$\xymatrix{
 & G \ar[d]^{\phi}
\\
B\ar[r]^{\alpha} & A,}\quad\quad\quad
\xymatrix{
 & G \ar[d]^{\phi}\ar@{.>}[ld]_{\widetilde{\phi}}
\\
B\ar[r]^{\alpha} & A,}
$$
where $A, B$ are finite groups, the solid arrows are continuous homomorphisms and $\alpha$ is surjective. A solution of this embedding map is a continuous homomorphism
$\widetilde{\phi}:G\to B$ which makes the right hand side diagram commutative.  We say that the embedding problem above is {\it real} if for every involution $t\in G$ with $\phi(t)\neq 1$ there is an involution $b\in B$ with $\alpha(b)=\phi(t)$. Following Haran and Jarden (see \cite{HJ}) we say that that a profinite group $G$ is {\it real projective} if $G$ has an open subgroup without $2$-torsion, and if every real embedding problem for $G$ has a solution. 
\end{defn}
\begin{thm}[Haran] A profinite group $G$ is real projective if and only if $G$ has an open subgroup $G_0$ of index $\leq 2$ with $\mathrm{cd}(G_0)\leq 1$, and every involution $t\in G$ is self-centralizing, that is, we have $C_G(t)=\{1, t\}$. 
\end{thm}
\begin{proof} This is Theorem A of \cite{Ha} on page 219.
\end{proof}
By classical Artin--Schreier theory every involution in the absolute Galois group of a field is self-centralising, so we get the following
\begin{cor} The absolute Galois group of a field $K$ is real projective if and only if $K$ satisfies $\mathrm{cd}(K(\mathbf i))\leq 1$.\qed
\end{cor}
We also need to recall Dwyer's theorem relating the vanishing of Massey products to certain embedding problems. 
\begin{defn}\label{11.2} Let $e_{ij}:\textrm{Mat}_n(\mathbb Z/p)\to\mathbb Z/p$ be the function taking an $n\times n$ matrix with coefficients in $\mathbb Z/p$ to its $(i, j)$-entry. Let
$$U_n(p)=\{U\in\textrm{Mat}_n(\mathbb Z/p)\mid e_{ii}(U)=1,\ e_{ij}(U)=0\  (\forall\ i>j)\}$$
be the group of upper triangular $n\times n$ invertible matrices with coefficients in $\mathbb Z/p$. Let $G$ and $\mathcal C^*$ be the same as in Definition \ref{1.2}. Then $H^1$ is naturally isomorphic to the group of continuous homomorphisms Hom$(G,\mathbb Z/p)$, and we will identify these two groups in all that follow. Given $n$ continuous homomorphisms
$$a_i:G\longrightarrow\mathbb Z/p
\quad(i=1,2,\ldots,n),$$
let $\mathbf E(a_1,a_2,\ldots,a_n)$ denote the embedding problem:
$$\xymatrix{
 & G\ar[d]^{-a_1\times-a_2\times\cdots\times-a_n}\ar@{.>}[ld]_{\psi}
\\
U_{n+1}(p)\ar[r]^{\phi_{n+1}} &
(\mathbb Z/p)^{n},}$$
where $\phi_{n+1}$ is given by the rule $U\mapsto(e_{12}(U),e_{23}(U),\ldots,e_{nn+1}(U))$.
\end{defn}
\begin{thm}[Dwyer]\label{dwyer} The $n$-fold Massey product $\langle a_1,a_2,\ldots,a_n\rangle$ vanishes if and only if the embedding problem
$\mathbf E(a_1,a_2,\ldots,a_n)$ has a solution.
\end{thm}
\begin{proof} See Theorem 2.4 of \cite{Dw} on page 182.
\end{proof}
\begin{rem}\label{compare} For every positive integer $m$ let $Z_m(p)$ and $P_m(p)$ denote the following subgroups of $U_m(p)$:
$$Z_m(p)=\{B\in U_m(p)\mid e_{ij}(B)=0\  
\textrm{if $1\leq i<j\leq m-1$ or $2\leq i<j\leq m$}\},$$
$$P_m(p)=\{B\in U_m(p)\mid e_{ij}(B)=0\  
\textrm{if $j=i+1,i+2$ and $1\leq i,j\leq m$}\},$$
respectively. Clearly $Z_m(p)\subset P_m(p)$ and they are different when $m>4$. By Theorem 2.4 of \cite{Dw} on page 182 quoted above the $n$-fold Massey product $\langle a_1,a_2,\ldots,a_n\rangle$ is defined if and only if the embedding problem:
$$\xymatrix{
 & G\ar[d]^{-a_1\times-a_2\times\cdots\times-a_n}\ar@{.>}[ld]_{\overline{\zeta}}
\\
U_{n+1}(p)/Z_{n+1}(p)\ar[r]^{\quad\quad\zeta_{n+1}} &
(\mathbb Z/p)^{n},}$$
has a solution, where $\zeta_{n+1}:U_{n+1}(p)/Z_{n+1}(p)\to(\mathbb Z/p)^{n}$ is the unique homomorphism such that the composition of the quotient map $U_{n+1}(p)\to U_{n+1}(p)/Z_{n+1}(p)$ and $\zeta_{n+1}$ is the homomorphism $\phi_{n+1}:U_{n+1}(p)\to(\mathbb Z/p)^{n}$ in Definition \ref{11.2}, while $a_i\cup a_{i+1}=0$ for every $1\leq i<n$ if and only if the embedding problem:
$$\xymatrix{
 & G\ar[d]^{-a_1\times-a_2\times\cdots\times-a_n}\ar@{.>}[ld]_{\overline{\kappa}}
\\
U_{n+1}(p)/P_{n+1}(p)\ar[r]^{\quad\quad\kappa_{n+1}} &
(\mathbb Z/p)^{n},}$$
has a solution, where $\kappa_{n+1}:U_{n+1}(p)/P_{n+1}(p)\to(\mathbb Z/p)^{n}$ is the unique group homomorphism such that the composition of the quotient map $U_{n+1}(p)\to U_{n+1}(p)/P_{n+1}(p)$ and $\kappa_{n+1}$ is the homomorphism $\phi_{n+1}$ in Definition \ref{11.2}.
\end{rem}
\begin{proof}[Proof of Theorem \ref{massey_strong}] By Haran's theorem it will be enough to show that every such embedding problem with $a_i\cup a_{i+1}=0$ for every $i=1,2,\ldots,n-1$ is real, in other words the restriction of the embedding problem to any subgroup of order two has a solution. In other words, by Artin-Schreier theory, we reduced the claim to the case when $K$ is real closed, i.e.~when $G(K)=\mathbb Z/2$. The claim for the latter is trivial when $p$ is odd, since $U_{n+1}(p)$ is a $p$-group. So we may assume without the loss of generality that $p=2$.

Now let $a_1,a_2,\ldots,a_n\in\textrm{Hom}(G(K),\mathbb Z/2)$ be a set of cohomology classes in $H^1$ such that $a_i\cup a_{i+1}=0$ for every $i=1,2,\ldots,n-1$. Let $g\in G(K)$ be the generator.
\begin{lemma}\label{case-by-case} For every $i=1,2\ldots,n-1$ the following holds: if $a_i(g)=1$ then $a_{i+1}(g)=0$.
\end{lemma}
\begin{proof} Since for every $a,b\in H^1$ the $2$-fold Massey product $\langle a,b\rangle$ is the singleton $a\cup b$, by Theorem \ref{dwyer} we get that the embedding problem:
$$\xymatrix{
 & G(K)\ar[d]^{a_i\times a_{i+1}}\ar@{.>}[ld]_{\psi_i}
\\
U_{3}(2)\ar[r]^{\phi_3} &
(\mathbb Z/2)^{2},}$$
has a solution $\psi_i$. Assume now that $a_i(g)=a_{i+1}(g)=1$. Then either
$$\psi_i(g) = \begin{pmatrix}
1 & 1 & 0 \\
0 & 1 & 1  \\
0 & 0 & 1 \end{pmatrix}\quad\textrm{or}\quad
\psi_i(g) = \begin{pmatrix}
1 & 1 & 1 \\
0 & 1 & 1  \\
0 & 0 & 1 \end{pmatrix}.$$
However neither of these matrices has order two, which is a contradiction.
\end{proof}
In plain English the lemma above means that we can break up the row vector
$$\begin{pmatrix}
a_1(g) & a_2(g) & \ldots & a_n(g) \end{pmatrix}$$
to single entries of $1$-s separated by zeros. By the above it is enough to construct a matrix $A\in U_{n+1}(p)$ such that $A^2$ is the identity matrix, and
$$\phi_{n+1}(A)=a_1\times a_2\times\cdots\times a_n(g).$$
In fact the matrix $A=(a_{ij})_{i,j=1}^{n+1}$ with
$$a_{ij}=\begin{cases}
1,&\text{if $i=j$,}\\
1,&\text{if  $i+1=j$ and $a_i(g)=1$,}\\
0,&\text{otherwise,} \end{cases}$$
will do. Indeed it is a block matrix whose off-diagonal terms are zero matrices, and the diagonal terms are either the $1\times 1$ matrix $(1)$ or the $2\times 2$ matrix $(\begin{smallmatrix}
1 & 1  \\ 0 & 1 \end{smallmatrix})$. These have order dividing two, so the same holds for $A$, too. 
\end{proof}
Finally we point out that Min\'a\v c and T\^an actually proved the strong Massey vanishing conjecture for what they call odd rigid fields.
\begin{defn} We say that a field $K$ is $p$-rigid for every $\alpha,\beta
\in H^1(G,\mathbb Z/p)$ such that $\alpha\cup\beta=0$, the linear subspace
$\mathrm{span}(\alpha,\beta)$ is at most one-dimensional. 
\end{defn}
\begin{thm}[Min\'a\v c--T\^an]\label{mintan} Let $p$ be an odd prime number and let $K$ be a $p$-rigid field which contains a primitive $p$-th root of unity. Then the strong Massey vanishing conjecture holds for $K$.
\end{thm}
This result has been essentially proved in \cite{MT2a} (see Theorem 8.5 of {\it loc.~cit.}~and its proof), however the authors only stated that the weak Massey vanishing conjecture holds for $K$. A minimal modification of the authors' argument will give this stronger result. We present the modified proof for the reader's convenience. 
\begin{proof} We are going to show the claim for all $n\geq2$ by induction on $n$. The initial case $n=2$ is trivially true. Let's assume that $n\geq3$ and the claim holds for $n-1$. Suppose first that there is an index $k\in\{1,2,\ldots,n\}$ such that $a_k=0$. If $k>1$ there is a homomorphism $\overleftarrow{\phi}:G\to U_k(p)$ lifting $-a_1\times-a_2\times\cdots\times-a_{k-1}$ by the induction hypothesis. Otherwise let $\overleftarrow{\phi}:G\to U_{1}(p)=\{1\}$ be the trivial homomorphism. If $k<n$ there is a homomorphism $\overrightarrow{\phi}:G\to U_{n-k}(p)$ lifting $-a_{k+1}\times-a_{k+2}\times\cdots\times-a_n$ by the induction hypothesis. Otherwise let $\overrightarrow{\phi}:G\to U_{1}(p)$ be the trivial homomorphism.

Now let $\widetilde{\phi}:G\to U_{n+1}(p)$ be the unique homomorphism such that
$$e_{ij}(\widetilde{\phi}(g))=\begin{cases}
    e_{ij}(\overleftarrow{\phi}(g)), & \text{if $1\leq i,j\leq k$} ,\\
    e_{(i-k)(j-k)}(\overrightarrow{\phi}(g)), & \text{if $k+1\leq i,j\leq n+1$} ,\\
    0, & \text{otherwise, }
    \end{cases}$$
for every $g\in G$. In plain English $\widetilde{\phi}(g)$ is a block matrix whose off-diagonal terms are zero matrices, and the diagonal terms are the $k\times k$ matrix $\overleftarrow{\phi}(g)$ and the $(n-k+1)\times(n-k+1)$ matrix
$\overrightarrow{\phi}(g)$. Clearly $\widetilde{\phi}$ is a lift of $-a_1\times-a_2\times\cdots\times-a_n$. So we may assume without the loss of generality that $a_k\neq0$ for every $k\in\{1,2,\ldots,n\}$. 

Then there is an $a\in H^1(G,\mathbb Z/p)$ such that $a_k=\lambda_ka$ for some $\lambda_k\in\mathbb Z/p$ for every $k$, since $K$ is $p$-rigid. By Theorem 8.1 of \cite{MT2a} the Massey product $\langle a,a,\ldots,a\rangle$ is defined and contains zero. Now a repeated application of part $(b)$ of Lemma 6.2.4 in \cite{Fe} on page 236 concludes the proof.  
\end{proof}

\section{Demushkin groups, $p$-adic embedding problems and Massey products}

\begin{defn} The {\it kernel} $\mathrm{Ker}(\mathbf E)$ of an embedding problem $\mathbf E$ as one in Definition \ref{2.1} is the kernel of $\alpha$. We say that $\mathbf E$ is {\it central} if $B$ is a central extension of $A$, that is, when $\mathrm{Ker}(\mathbf E)$ lies in the centre of $B$. In this case $\mathrm{Ker}(\mathbf E)$ is abelian, so we can equip it with the trivial $G$-module structure. 
\end{defn}
\begin{defn} Assume now that the embedding problem $\mathbf E$ is central. The {\it obstruction class of $\mathbf E$} is defined as follows. Let
$\widehat{\phi}:G\to B$ be a continuous map such that $\alpha\circ
\widehat{\phi}=\phi$. Then the map $c:G\times G\to\mathrm{Ker}(\mathbf E)$ given by the rule:
$$c(x,y)=\widehat{\phi}(xy)\widehat{\phi}(y)^{-1}\widehat{\phi}(x)^{-1}\in
\mathrm{Ker}(\mathbf E),\quad(x,y\in G)$$
is a cocycle, and its cohomology class $o(\mathbf E)\in H^2(G,\mathrm{Ker}(\mathbf E))$ does not depend on the choice of $\widehat{\phi}$, only on
$\mathbf E$. By a well-known classical result $\mathbf E$ has a solution if and only if $o(\mathbf E)=0$. 
\end{defn}
\begin{lemma}\label{easy_vanishing} Let $G$ be a profinite group such that $H^2(G,\mathbb Z/p)=0$. Then $G$ has the strong vanishing $n$-fold Massey product property with respect to $p$.
\end{lemma}
\begin{proof} Since $U_{n+1}(p)$ is a $p$-group, it has a filtration by normal subgroups:
$$\{1\}=N_0\subset N_1\subset\cdots\subset N_{\binom{n-1}{2}}=\mathrm{Ker}(\phi)$$
such that $U_{n+1}(p)/N_k$ is a central extension of $U_{n+1}(p)/N_{k+1}$ and the kernel of the quotient map $\pi_k:U_{n+1}(p)/N_k\to U_{n+1}(p)/N_{k+1}$ is:
$$\frac{N_{k+1}}{N_k}\cong\mathbb Z/p$$
for every $k=0,1,\ldots,\binom{n-1}{2}-1$.

Note that it will be sufficient to show that the embedding problem
$\mathcal E(h)$:
$$\xymatrix{
 & G \ar[d]^{h}\ar@{.>}[ld]_{\overline h}
\\
U_{n+1}(p)/N_k\ar[r]^{\!\!\pi_k} &
U_{n+1}(p)/N_{k+1}}$$
for every every homomorphism $G\to U_{n+1}(p)/N_{k+1}$ has a solution for every $k=0,1,\ldots,\binom{n-1}{2}$. Indeed then we would get by descending induction on the index $k$ that $-a_1\times\cdots\times-a_n$ has a lift to $G\to U_{n+1}(p)/N_k$. The claim is now clear from the case $k=0$. 

However $\mathcal E(h)$ is a central embedding problem with kernel isomorphic to $\mathbb Z/p$ by the above. So its obstruction class $o(\mathcal E(h))$ lies in $H^2(G,\mathbb Z/p)$, which is zero by assumption. So $o(\mathcal E(h))$ vanishes, and hence $\mathcal E(h)$ has a solution.
\end{proof}
\begin{defn} A pro-$p$ group $G$ is said to be a {\it Demushkin group} if
\begin{enumerate}
\item $\dim_{\mathbb Z/p}H^1(G,\mathbb Z/p)<\infty$,
\item $\dim_{\mathbb Z/p}H^2(G,\mathbb Z/p)=1$,
\item the cup product $H^1(G,\mathbb Z/p)\times H^1(G,\mathbb Z/p)\to H^2(G,\mathbb Z/p)$ is a non-degenerate bilinear form.
\end{enumerate}
\end{defn}
\begin{thm}\label{demushkin_vanishing} Let $n\geq 3$ be an integer and let $p$ be a prime number. Then every pro-$p$ Demushkin group has the strong vanishing $n$-fold Massey product property with respect to $p$.
\end{thm}
This claim above is a strengthening of Theorem 4.3 of \cite{MT2} on page 265, which in turn is a generalisation of Lemma 3.5 of \cite{HW} on page 1317. Note that $\mathbb Z/2$ is a Demushkin group, so the theorem above generalises the key ingredient of the proof of Theorem \ref{massey_strong}. 
\begin{proof}[Proof of Theorem \ref{demushkin_vanishing}]  Arguing the same way as we did at the beginning of the proof of Theorem \ref{mintan}, we may assume without the loss of generality that $a_k\neq0$ for every $k\in\{1,2,\ldots,n\}$. Let $M_{k,m}$ denote the subgroup
$$M_{k,m}=\{U\in U_m(p)\mid e_{ij}(U)=0\  \textrm{if $1\leq i<j\leq m-1$ or
$j=m,k\leq i\leq m-1$}\}$$
for every $k=1,2,\ldots,m-1$. Clearly $M_{k,m}\subset M_{k+1,m}$ for every $k=1,2,\ldots,m-2$. 
\begin{lemma}\label{normal} The subgroup $M_{k,m}$ of $U_m(p)$ is normal.
\end{lemma}
\begin{proof} For every pair $a\leq b$ of natural numbers let
$\overleftarrow{\beta}_{a,b}:U_b(p)\to U_a(p)$ be the homomorphism:
$$U\mapsto (e_{ij}(U))_{i,j=1}^a,$$
that is the map which assigns to every element of $U_b(p)$ its upper left $a\times a$ block. Similarly $\overrightarrow{\beta}_{a,b}:U_b(p)\to U_a(p)$ be the homomorphism:
$$U\mapsto (e_{(b-a+i)(b-a+j)}(U))_{i,j=1}^a,$$
that is the map which assigns to every element of $U_b(p)$ its lower right $a\times a$ block. The subgroup $M_{k,m}$ is the intersection of the kernel of
$\overleftarrow{\beta}_{m-1,m}$ and the kernel of $\overrightarrow{\beta}_{m+1-k,m}$, so as an intersection of normal subgroups, it is normal. 
\end{proof}
Let $Q_{k,m}$ denote the quotient group $U_m(p)/M_{k,m}$ and let
$\rho_{k,m}:Q_{k,m}\to Q_{k+1,m}$ denote the quotient map induced by the inclusion $M_{k,m}\subset M_{k+1,m}$.
\begin{lemma}\label{kernel} The extension $Q_{k,m}$ of $Q_{k+1,m}$ is central, and the kernel of $\rho_{k,m}$ is isomorphic to $\mathbb Z/p$ for every $k=1,2,\ldots,m-2$.
\end{lemma}
\begin{proof} As we saw in the proof of Lemma \ref{normal} above $Q_{k,m}$ is the fibre product:
$$\{(A,B)\in U_{m-1}(p)\times U_{m+1-k}(p)
\mid \overrightarrow{\beta}_{m-k,m-1}(A)=
\overleftarrow{\beta}_{m-k,m+1-k}(B)\},$$
considered as a subgroup of $U_{m-1}(p)\times U_{m+1-k}(p)$. Under this identification the group $\mathrm{Ker}(\rho_{k,m})$ is:
$$\{(A,B)\in U_{m-1}(p)\times U_{m+1-k}(p)\mid
A=I_{m-1\times m-1},\ B\in Z_{m+1-k}(p)\},$$
where $I_{m-1\times m-1}$ is the identity matrix and $Z_{m+1-k}(p)$ is the group defined in Remark \ref{compare}. Since $Z_{m+1-k}(p)$ is the centre of $U_{m+1-k}(p)$, it lies in the centre of $Q_{k,m}$, so the extension $Q_{k,m}$ of $Q_{k+1,m}$ is central. The map $\iota_{k,m}:\mathrm{Ker}(\rho_{k,m})\to
\mathbb Z/(p)$ which is the composition of the isomorphism $\mathrm{Ker}(\rho_{k,m})\to Z_{m+1-k}(p)$ given by the rule $(A,B)\mapsto B$, and the map $Z_{m+1-k}(p)\to\mathbb Z/p$ given by the rule $B\mapsto e_{1(m+1-k)}(B)$, is an isomorphism. 
\end{proof}
We will identify the group $\mathrm{Ker}(\rho_{k,m})$ with $\mathbb Z/p$ via the isomorphism $\iota_{k,m}$ in the proof of Lemma \ref{kernel} in all that follows. For every homomorphism $\psi:G\to Q_{k+1,m}$ let $E(\psi)$ denote the embedding problem:
$$\xymatrix{
 & G \ar[d]^{\psi}\ar@{.>}[ld]_{\overline{\psi}}
\\
Q_{k,m}\ar[r]^{\!\rho_{k,m}} & Q_{k+1,m}.}$$
Let $\chi:G\to\mathrm{Ker}(\rho_{k+1,m})$ be a homomorphism. Then the map
$\psi\chi$ given by the rule $g\mapsto\psi(g)\chi(g)$ is also a homomorphism from $G$ to $Q_{k+1,m}$, since $\mathrm{Ker}(\rho_{k+1,m})$ lies in the centre of $Q_{k+1,m}$ by Lemma \ref{kernel}. Let $\phi_{k,m+1}:Q_{k,m+1}
\to(\mathbb Z/p)^{m}$ be the unique homomorphism such that the composition of the quotient map $U_{m+1}(p)\to Q_{k,m+1}$ and $\phi_{k,m+1}$ is the homomorphism 
$\phi_{m+1}:U_{m+1}(p)\to(\mathbb Z/p)^{m}$ in Definition \ref{11.2} for every $k=1,2,\ldots,m-1$. 
\begin{lemma}\label{twisting} Let $\psi:G\to Q_{k+1,n+1}$ be a homomorphism such that $\phi_{k+1,n+1}\circ\psi=-a_1\times-a_2\times\cdots\times-a_n$. Then 
$$o(E(\psi\chi))=o(E(\psi))+a_k\cup\chi$$
in $H^2(G,\mathbb Z/p)$ for every homomorphism $\chi:G\to\mathrm{Ker}(\rho_{k+1,n+1})$. 
\end{lemma}
\begin{proof} There is a commutative diagram of group homomorphisms:
$$\xymatrix{
Q_{k,n+1}\ar[d]\ar[r]^{\!\!\!\!\!\rho_{k,n+1}} & Q_{k+1,n+1} \ar[d]^{\lambda}
\\ Q_{1,n-k+2}\ar[r]^{\!\!\rho_{1,n-k+2}} & Q_{2,n-k+2}}$$
such that the vertical arrow on the left induces an isomorphism between
$\mathrm{Ker}(\rho_{k,n+1})$ and $\mathrm{Ker}(\rho_{1,n-k+2})$ and
$\phi_{2,n-k+2}\circ\lambda\circ\psi=-a_k\times-a_{k+1}\times\cdots\times-a_n$. Therefore we may assume that $k=1$ without the loss of generality, because the obstruction classes of central embedding problems are natural. In this case the obstruction classes $o(E(\psi))$ and $o(E(\psi\chi))$ are given by the Massey products for the defining systems corresponding to $\psi$ and $\psi\chi$, respectively, by Theorem 2.4 and the remark immediately follow it in \cite{Dw} on page 182. Therefore the difference between the two is $a_1\cup\chi$ (compare with Remark 2.2 of \cite{MT2} on page 261), and hence the claim follows. 
\end{proof}
Now we are going to prove for every $k=1,2,\ldots,n-1$ that the embedding problem
$\mathbb E(k)$:
$$\xymatrix{
& G \ar[d]^{-a_1\times-a_2\times\cdots\times-a_n}\ar@{.>}[ld]_{\psi}
\\ Q_{k,n+1}\ar[r]^{\!\phi_{k,n+1}} & (\mathbb Z/p)^{n},}$$
has a solution by descending induction on $k$. Since the case $k=1$ is the claim, this will be sufficient to conclude the proof.

Let us first consider the case $k=n-1$ first. By the induction hypothesis there are solutions $\overleftarrow{\psi}$ and $\overrightarrow{\psi}$ to the embedding problems:
$$\xymatrix{
& G \ar[d]^{-a_1\times-a_2\times\cdots\times-a_{n-1}}\ar@{.>}[ld]_{\overleftarrow{\psi}}
\\ U_n(p) \ar[r]^{\!\!\!\phi_{n}} & (\mathbb Z/p)^{n-1},}
\quad\quad\quad
\xymatrix{
& G \ar[d]^{-a_{n-1}\times-a_n}\ar@{.>}[ld]_{\overrightarrow{\psi}}
\\ U_3(p) \ar[r]^{\!\!\phi_3} & (\mathbb Z/p)^2,}$$
respectively. The direct product $\overleftarrow{\psi}\times\overrightarrow{\psi}
:G\to U_n(p)\times U_3(p)$ lies in $Q_{n-1,n+1}$, and it is a solution to
$\mathbb E(n-1)$. 

Now assume that $\mathbb E(k)$ has a solution $\psi$ for some $k\geq2$. By assumption $a_{k-1}\neq0$ and the cup product
$$\cup:H^1(G,\mathbb Z/p)\times H^1(G,\mathbb Z/p)\to
H^2(G,\mathbb Z/p)\cong \mathbb Z/p$$
is a non-degenerate bilinear form, so there is a $\chi\in H^1(G,\mathbb Z/p)$
such that
$$o(E(\psi))+a_{k-1}\cup\chi=0.$$
Therefore there is a solution $\overline{\psi}:G\to Q_{k-1,n+1}$ to $E(\psi\chi)$ by Lemma \ref{twisting}. This $\overline{\psi}$ is also a solution to $\mathbb E(k-1)$, since $\phi_{k-1,n+1}=\phi_{k,n+1}\circ\rho_{k-1,n+1}$.  
\end{proof}

\begin{lemma}\label{labute} For every pair of prime numbers $l,p$, not necessarily different, the maximal $l$-adic quotient $H$ of $G(\mathbb Q_p)$ is either a pro-$l$ Demushkin group or we have $H^2(H,\mathbb Z/l)=0$. 
\end{lemma}
\begin{proof} When $l=p$ then either $H$ is projective, and hence $H^2(H,\mathbb Z/l)=0$, or it is a Demushkin group (see \S5 of \cite{La} on pages 130-31 for a proof). When $l\neq p$ then $H$ is $\mathbb Z_l$ if $l$ does not divide $p-1$, and hence $H^2(H,\mathbb Z/l)=0$, or $H$ is the semi-direct product $\mathbb Z_l\rtimes\mathbb Z_l$ such that the second copy of
$\mathbb Z_l$ acts on the first via the character $m\mapsto p^m$ if $l$ divides $p-1$, and hence it is a Demushkin group. 
\end{proof}
\begin{defn} We say that an embedding problem as the one in Definition \ref{2.1} above is a {\it $G(\mathbb Q_p)$-problem} if for every closed subgroup $H$ of $G$ which is isomorphic to $G(\mathbb Q_p)$ there is a homomorphism
$\widetilde{\phi}_H:H\to B$ such that $\alpha\circ\widetilde{\phi}_H=\phi|_H$. Following Haran and Jarden (see \cite{HJ2}) we say that that a profinite group $G$ is {\it $p$-adically projective} if every $G(\mathbb Q_p)$-problem for $G$ has a solution, and if the collection of all closed subgroups of $G$ which are isomorphic to $G(\mathbb Q_p)$ is topologically closed. 
\end{defn}
By the main result of \cite{HJ2} (see the Theorem on page 148) we know the following
\begin{thm}[Haran--Jarden]\label{hj2} If K is a $\mathrm{PpC}$ field, then $G(K)$ is $p$-adically projective. Conversely, if $G$ is a $p$-adically projective group, then there exists a $\mathrm{PpC}$ field $K$ such that $G(K)\cong G$.\qed
\end{thm}
Now we are ready to give a 
\begin{proof}[Proof of Theorem \ref{main_p-adic}] By the Haran--Jarden Theorem \ref{hj2} it will be enough to show that for every prime number $l$ every such embedding problem with $a_i\cup a_{i+1}=0$ for every $i=1,2,\ldots,n-1$ is a $G(\mathbb Q_p)$-problem. In order to do so it will be enough to show that for every $l$ as above the maximal $l$-adic quotient of $G(\mathbb Q_p)$ has the strong vanishing $n$-fold Massey product property with respect to $l$. However this is immediate from Lemma \ref{labute}, Lemma \ref{easy_vanishing} and Theorem \ref{demushkin_vanishing}. 
\end{proof}

\section{Properties of unipotent groups}

\begin{defn} As usual let $E_{ij}\in\mathrm{Mat}_n(\mathbb Z/p)$ denote the elementary matrix characterised by the property that
$$e_{kl}(E_{ij})=\begin{cases}
1,&\text{if $k=i$ and $l=j$,}\\
0,&\text{otherwise,} \end{cases}$$
and let $I\in\mathrm{Mat}_n(\mathbb Z/p)$ be the identity matrix. They satisfy the following identity:
\begin{equation}\label{0.4.1} E_{ij}E_{kl}=\begin{cases}
E_{il},&\text{if $j=k$,}\\
0,&\text{otherwise.} \end{cases} \end{equation}
\end{defn}
\begin{lemma}\label{identity} For every $i<j$ and $k<l$ the following hold:
\begin{enumerate}
\item[$(a)$] we have $E_{ij}^2=0$,
\item[$(b)$] we have $E_{ij}E_{kl}E_{ij}=0$, 
\item[$(c)$] we have:
$$[E_{ij},E_{kl}]=\begin{cases}
E_{il},&\text{if $j=k$,}\\
-E_{kj},&\text{if $l=i$,}\\
0,&\text{otherwise,} \end{cases}$$
\end{enumerate}
\end{lemma}
\begin{proof} Since $i\neq j$, part $(a)$ is immediate from (\ref{0.4.1}). If $E_{ij}E_{kl}E_{ij}\neq0$ then $j=k$ and $l=i$ by (\ref{0.4.1}). Then $i<j=k<l=i$, which is a contradiction. Therefore part $(b)$ holds. If in 
$$[E_{ij},E_{kl}]=E_{ij}E_{kl}-E_{kl}E_{ij}$$
both terms are non-zero, then $j=k$ and $l=i$ by (\ref{0.4.1}). This is not possible, as we have just seen. Therefore at most one of the terms is non-zero, and hence part $(c)$ claim follows. 
\end{proof}
\begin{lemma}\label{commutator} We have:
$$[I+E_{ij},I+E_{kl}]=\begin{cases}
I+E_{il},&\text{if $j=k$,}\\
I-E_{kj},&\text{if $l=i$,}\\
I,&\text{otherwise,} \end{cases}$$
for every $i<j$ and $k<l$.
\end{lemma}
\begin{proof} Note that
$$(I+E_{ij})(I-E_{ij})=I-E_{ij}+E_{ij}-E_{ij}^2=I$$
when $i<j$, using claim $(a)$ of Lemma \ref{identity}. Therefore
\begin{align*}
[I+E_{ij},I+E_{kl}] = &  (I+E_{ij})(I+E_{kl})(I-E_{ij})(I-E_{kl})  \\
 =  & I+E_{ij}+E_{kl}+E_{ij}E_{kl}-E_{ij}-{\color{red} E_{ij}^2}-E_{kl}E_{ij}-{\color{blue} E_{ij}E_{kl}E_{ij}} \\
  &  -E_{kl}-E_{ij}E_{kl}-{\color{red} E_{kl}^2}+{\color{red} E_{ij}E_{kl}^2} \\
  & +E_{ij}E_{kl}+{\color{red} E_{ij}^2E_{kl}}
+{\color{blue} E_{kl}E_{ij}E_{kl}}+{\color{blue} E_{ij}E_{kl}E_{ij}E_{kl}}.
\end{align*}
By part $(a)$ of Lemma \ref{identity} all red terms are zero, while by part $(b)$ of Lemma \ref{identity} all blue terms are zero. Therefore 
$$[I+E_{ij},I+E_{kl}]=I+E_{ij}E_{kl}-E_{kl}E_{ij}=I+[E_{ij},E_{kl}].$$
because of the cancellations between the remaining terms. The claim follows from part $(c)$ of Lemma \ref{identity}.
\end{proof}
\begin{notn} For every positive integer $m$ let $K_m(p)$ denote the following subgroup of $U_m(p)$:
$$K_m(p)=\{B\in U_m(p)\mid e_{ii+1}(B)=0\  
\textrm{if $1\leq i<m$}\}.$$
For every pair of positive integers $k,m$ let $U_{k,m}(p)$ denote the following subgroup of $U_m(p)$:
$$U_{k,m}(p)=\{B\in U_m(p)\mid e_{ij}(B)=0\  
\textrm{if $1\leq i<j\leq\min(i+k-1,m)$}\},$$
Note that, using this notation, we have
$$Z_m(p)=U_{m-1,m}(p),\  P_m(p)=U_{3,m}(p),\ U_m(p)=U_{1,m}(p)\textrm{ and }K_m(p)=U_{2,m}(p).$$ 
Finally let $\iota_m:\mathbb Z/p\to Z_m(p)$ be the unique isomorphism such that
$$\iota_m(1)=I+E_{1m}.$$ 
\end{notn}
\begin{defn}\label{blocks} For every $1\leq i<j\leq m$ let $\mathbf b_{ij}:U_{m}(p)\to U_{j-i+1}(p)$ be the unique homomorphism such that
$$e_{kl}(\mathbf b_{ij}(B))=e_{(i+k-1)(i+l-1)}(B)\quad(\forall B\in U_{m}(p),\ 1\leq k<l\leq j-i+1).$$
In plain English this is the $(j-i+1)\times(j-i+1)$ diagonal block with the right upper corner at the $(i,j)$-th entry. Note that
\begin{equation}\label{0.8.1} 
 U_{k,m}(p)=\bigcap_{\substack{ j-i=k-1 \\ 1\leq i<j\leq m } }
 \mathrm{Ker}(\mathbf b_{ij})\quad(\forall2\leq k\leq m-1).
  \end{equation}
In particular $U_{k,m}(p)$ is a normal subgroup of $U_m(p)$ for every $k$. 
\end{defn}
\begin{prop}\label{generation} The subgroup $U_{k,m}(p)$ is generated by the elements:
$$\{I+E_{ij}\mid k\leq i+k-1<j\leq m\}.$$
\end{prop}
\begin{proof} First note that these elements are actually in $U_{k,m}(p)$, so the claim actually makes sense. We are going to show the latter by descending induction on $k$. When $k\geq m$ then $U_{k,m}(p)$ is the trivial group so the claim is trivially true. Now assume that the claim is true for $1\leq k+1\leq m$. Then it will be sufficient to prove the lemma below.
\end{proof}
\begin{lemma}\label{generation2} The image of the set $\{I+E_{ii+k}\mid 1\leq i\leq m-k\}$ with respect to the quotient map $U_{k,m}(p)\to U_{k,m}(p)/U_{k+1,m}(p)$ is a basis of the $p$-torsion abelian group  $U_{k,m}(p)/U_{k+1,m}(p)$. 
\end{lemma}
\begin{proof} Recall that the kernel of the map:
$$\mathbf b_{k,m}=\prod_{\substack{ j-i=k \\ 1\leq i<j\leq m }}
\mathbf b_{ij}: U_{k,m}(p)\to 
Z_m(p)^{m-k}\cong(\mathbb Z/p)^{m-k}$$
is $U_{k+1,m}(p)$. Moreover under the identification $Z_m(p)^{m-k}\cong(\mathbb Z/p)^{m-k}$ furnished by $\iota_m$ this homomorphism
$\mathbf b_{k,m}$ maps the set $\{I+E_{ii+k}\mid 1\leq i\leq m-k\}$ bijectively onto the standard basis of $(\mathbb Z/p)^{m-k}$ for every $1\leq k\leq m-1$. The claim is now clear. 
\end{proof}
\begin{cor}\label{derived} The derived subgroup $K_m(p)'$ is $U_{4,m}(p)$.
\end{cor}
\begin{proof} For every $i,j$ with $4\leq i+3<j\leq m$ we have
$$I+E_{ij}=[I+E_{i(i+2)},I+E_{(i+2)j}]\in K_m(p)'$$
by Lemma \ref{commutator}. Therefore $K_m(p)'$ contains $U_{4,m}(p)$ by Proposition \ref{generation}. On the other hand the quotient $K_m(p)/U_{4,m}(p)$ is generated by the images $J$ of
$$\{I+E_{ij}\mid 2\leq i+1<j\leq m,\ j\leq i+3\}$$
under the quotient map $K_m(p)\to K_m(p)/U_{4,m}(p)$ by Proposition \ref{generation}. Since for $i,j$ and $k,l$ with $2\leq i+1<j\leq m,\ j\leq i+3$ and $2\leq k+1<l\leq m,\ l\leq k+3$ if $j=k$, then $i+3<l$, and if $l=i$, then $k+3<j$, so 
we have
$$[I+E_{ij},I+E_{kl}]\in U_{4,m}(p)$$
using Lemma \ref{commutator}. Hence the elements of $J$ commute. Since they generate the quotient $K_m(p)/U_{4,m}(p)$ we get that the latter is commutative. Therefore $U_{4,m}(p)$ contains $K_m(p)'$, too. 
\end{proof}
\begin{defn} Recall that the $p$-Zassenhaus filtration of a finite $p$-group $G$, denoted by $G_{(n,p)},\ n =1,2,\ldots$, is defined inductively by
$$G_{(1,p)}=G,\quad G_{(n,p)}=(G_{(\lceil n/p\rceil,p)})^p\prod_{i+j=n}
[G_{(i,p)},G_{(j,p)}]\quad\textrm{for $n\geq2$.}$$
(The original definition is different, but it is equivalent to this one by a theorem of Lazard, see Theorem 11.2 of \cite{DDMS} on page 271). As its name suggest this is a descending filtration by characteristic subgroups. It follows from the formula above that for every $n$ the quotient $G_{(n,p)}/G_{(n+1,p)}$ is abelian of exponent dividing $p$. Consider the graded $\mathbb Z/p$-module:
$$\mathrm{gr}(G)=\bigoplus_{n\geq0}G_{(n,p)}/G_{(n+1,p)}.$$
The commutator map and the $p$-power map induce on $\textrm{gr}(G)$ the structure of a $p$-restricted Lie $\mathbb Z/p$--algebra (see \S12.2 of \cite{DDMS} on pages 298-305). 
\end{defn}
\begin{prop}\label{zassenhaus} We have $U_m(p)_{(k,p)}=U_{k,m}(p)$ for every $k,m\geq1$. 
\end{prop}
\begin{proof} First we are going to show that $I+E_{ij}\in U_m(p)_{(j-i,p)}$ for every $i,j$ such that $1\leq i<j\leq m$ by induction on $j-i$. The case $j-i=1$ is trivial. Now let's assume that the claim is true when $j-i=k-1$, where $k\geq2$. Then for every $1\leq i<j\leq m$ such that $j-i=k$ we have
$$I+E_{ij}=[I+E_{ii+1},I+E_{i+1j}]\in [U_m(p)_{(1,p)},U_m(p)_{(j-i-1,p)}]\subseteq U_m(p)_{(j-i,p)}$$ 
using Lemma \ref{commutator} and the induction hypothesis. Since by Proposition \ref{generation} the subgroup $U_{k,m}(p)$ is generated by the elements
$$\{I+E_{ij}\mid k\leq i+k-1<j\leq m\},$$
we get that $U_m(p)_{(k,p)}\supseteq U_{k,m}(p)$ for every $k,m$. We are going to show the reverse inclusion by induction on $m$. The case $m=1$ is trivial. Now let's assume that $m\geq2$ and the claim is already known for $m-1$. Note that for every group homomorphism $\alpha:G\to H$ we have $\alpha(G_{(k,p)})\subseteq H_{(k,p)}$ for every $k\geq1$. Therefore
\begin{align*}
U_m(p)_{(k,p)}\subseteq &  \mathbf b_{1m-1}^{-1}(U_{m-1}(p)_{(k,p)})\cap
\mathbf b_{2m}^{-1}(U_{m-1}(p)_{(k,p)}) \\
 =  & \mathbf b_{1m-1}^{-1}(U_{k,m-1}(p))\cap
 \mathbf b_{2m}^{-1}(U_{k,m-1}(p))=U_{k,m}(p)
\end{align*}
when $k<m$, using the induction hypothesis. Hence $U_m(p)_{(k,p)}= U_{k,m}(p)$ when $k<m$. In order to complete the proof, we only need to show the following
\begin{lemma}\label{shiftit} We have:
\begin{enumerate}
\item[$(a)$] we have $[U_{(k,p)}(p),U_{(l,p)}(p)]=\{1\}$ when $k+l=m$, 
\item[$(b)$] we have $U_{\lceil m/p\rceil,m}(p)^p=\{1\}$. 
\end{enumerate}
\end{lemma}
Indeed $[U_m(p)_{(k,p)},U_m(p)_{(l,p)}]=[U_{k,m}(p),U_{l,m}(p)]$ by the above, as both $k<m$ and $l<m$. Moreover $\lceil m/p\rceil<m$, since $m\geq2$, so $U_m(p)_{(\lceil m/p\rceil,p)}^p=U_{\lceil m/p\rceil,m}(p)^p$. Since $U_m(p)_{(m,p)}$ is generated by these groups, we get that this group is trivial. Therefore all the higher terms in the $p$-Zassenhaus filtration of $U_m(p)$ are trivial, too. Since $U_{k,m}$ is trivial when $k\geq m$, the claim also holds when $k\geq m$. 
\end{proof}
\begin{proof}[Proof of Lemma \ref{shiftit}] Let $A\in U_{k,m}(p)$ and $B\in U_{l,m}(p)$ be arbitrary. Write $A=I+X$ and $B=I+Y$ such that 
$e_{ij}(X)=0$, if $1\leq j\leq\min(i+k-1,m)$, and $e_{ij}(Y)=0$, if $1\leq j\leq\min(i+l-1,m)$. For every row vector
$\alpha=(\alpha_1,\alpha_2,\ldots,\alpha_m)\in(\mathbb Z/p)^m$ let $\mathbf e_l(\alpha)$ be its $l$-th coordinate $\alpha_l$. If $\alpha\in(\mathbb Z/p)^m$ is a row vector such that $\mathbf e_i(\alpha)=0$ when $i\leq j$ for some $j=0,1,\ldots$ then $\mathbf e_i(\alpha X)=0$ when $i\leq j+k$. Similarly
$\mathbf e_i(\alpha Y)=0$ when $i\leq j+l$. Therefore
$\alpha XY$ is zero for every row vector $\alpha$, since $k+l=m$. Hence $XY=0$, and so $AB=(I+X)(I+Y)=I+X+Y$. A similar argument shows that $BA=I+Y+X$. Therefore $AB=BA$, so claim $(a)$ holds.

Now let $A\in U_{\lceil m/p\rceil,m}(p)$ be arbitrary, and write $A=I+X$ such that $e_{ij}(X)=0$, if $1\leq j\leq\min(i+\lceil m/p\rceil-1,m)$. If
$\alpha\in(\mathbb Z/p)^m$ is a row vector such that $\mathbf e_i(\alpha)=0$ when $i\leq j$ for some $j=0,1,\ldots$ then $\mathbf e_i(\alpha X)=0$ when $i\leq j+\lceil m/p\rceil$. We get that $\mathbf e_i(\alpha X^d)=0$ when $i\leq j+\lceil m/p\rceil d$ by induction on $d$. So $\alpha X^p$ is zero for every row vector $\alpha$, since $\lceil m/p\rceil p\geq m$. Hence $X^p=0$, and so $A^p=(I+X)^p=I^p+X^p=I$. Claim
$(b)$ follows. 
\end{proof}
\begin{notn} Let $\mathfrak{gl}_m(p)$ denote the Lie algebra associated to the rank $m$ matrix algebra $\mathrm{Mat}_m(\mathbb Z/p)$ over $\mathbb Z/p$. Let $\mathfrak u_m(p)\subset\mathfrak{gl}_m(p)$ denote the sub-Lie algebra of strictly upper triangular matrices: 
$$\mathfrak u_m(p)=\{B\in \mathfrak{gl}_m(p)\mid e_{ij}(B)=0\  
\textrm{if $1\leq j\leq i\leq m$}\}.$$
Since $\mathfrak u_m(p)$ is a Lie subalgebra of the Lie algebra of an associative algebra which is closed under the $p$-power map, it has the structure of a $p$-restricted Lie $\mathbb Z/p$--algebra. For every $A\in U_m(p)_{(k,p)}$ let $c_k(A)\in\mathrm{gr}(U_m(p))$ denote its class in the quotient $U_m(p)_{(k,p)}/U_m(p)_{(k+1,p)}$. In order to distinguish it from the commutator in groups, we will let $\llbracket,\rrbracket$ denote the Lie bracket in Lie algebras. 
\end{notn}
\begin{prop}\label{graded} There is a unique isomorphism
$$\lambda_m:\mathrm{gr}(U_m(p))\to\mathfrak u_m(p)$$
of $p$-restricted Lie $\mathbb Z/p$--algebras such that $\lambda_m(c_{j-i}(I+E_{ij}))=
E_{ij}$ for every $i,j$ such that $1\leq i<j\leq m$. 
\end{prop}
\begin{proof} By Lemma \ref{generation2} and Proposition \ref{zassenhaus} the set $\{c_k(I+E_{ii+k})\mid 1\leq i\leq m-k\}$ is a basis of $U_m(p)_{(k,p)}/U_m(p)_{(k+1,p)}$ for every $k\geq1$. Therefore
$$\{c_{j-i}(I+E_{ij})\mid 1\leq i<j\leq m\}$$
is a basis of $\mathrm{gr}(U_m(p))$. So there is a unique $\mathbb Z/p$-linear map $\lambda_m:\mathrm{gr}(U_m(p))\to\mathfrak u_m(p)$ such that $\lambda_m(c_{j-i}(I+E_{ij}))=E_{ij}$ for every $i,j$ such that $1\leq i<j\leq m$. Since $\lambda_m$ maps a basis onto a basis, it is an isomorphism. Since
$$\lambda_m(\llbracket c_{j-i}(I+E_{ij}),c_{l-k}(I+E_{kl})\rrbracket)=\llbracket\lambda_m(c_{j-i}(I+E_{ij})),\lambda_m(c_{l-k}(I+E_{kl}))
\rrbracket$$
for every $i<j$ and $k<l$ by part $(c)$ of Lemma \ref{identity} and by Lemma \ref{commutator}, we get that $\lambda_m$ is a Lie-algebra homomorphism using the bilinearity of the Lie bracket.

Let $(\cdot)^{[p]}$ denote  the $p$-operation of any $p$-restricted Lie $\mathbb Z/p$--algebra. Then we have $(I+E_{ij})^p=I$, so $c_{j-i}(I+E_{ij})^{[p]}=0$, while $E_{ij}^{[p]}=E_{ij}^p=0$, hence
$$\lambda_m(c_{j-i}(I+E_{ij})^{[p]})=\lambda_m(c_{j-i}(I+E_{ij}))^{[p]}$$
for every $1\leq i<j\leq m$. Now we only need to add the following well-known fact: if $\lambda:\mathfrak g\to\mathfrak h$ is a Lie algebra isomorphism between $p$-restricted Lie $\mathbb Z/p$--algebras, and it respects the $p$-operation on a basis of $\mathfrak g$, then it is an isomorphism between $p$-restricted Lie $\mathbb Z/p$--algebras. 
\end{proof}
\begin{defn} Let $\phi_{n+1}:U_{n+1}(p)\to(\mathbb Z/p)^{n}$ be the homomorphism given by the rule $U\mapsto(e_{12}(U),e_{23}(U),\ldots,e_{nn+1}(U))$. Similarly let $\eta_{n+1}:K_{n+1}(p)\to(\mathbb Z/p)^{n-1}$ be the homomorphism given by the rule $U\mapsto(e_{13}(U),e_{24}(U),\ldots,e_{n-1n+1}(U))$. Let $\{\cdot,\cdot\}:(\mathbb Z/p)^3\times
(\mathbb Z/p)^2\to\mathbb Z/p$ be the bilinear pairing given by the rule:
$$\{(a_1,a_2,a_3),(b_1,b_2)\}=a_1b_2-a_3b_1.$$
\end{defn}
\begin{cor}\label{u(3)ids} For every $A\in U_4(p)$ and $B\in K_4(p)$ we have:
$$[A,B]=\iota_4(\{\phi_4(A),\eta_4(B)\}).$$
\end{cor}
\begin{proof} From Proposition \ref{zassenhaus} we know that $A\in U_4(p)_{(1,p)}$ and $B\in U_4(p)_{(2,p)}$, so we get that $[A,B]\in U_4(p)_{(3,p)}=Z_4(p)$ using again Proposition \ref{zassenhaus}. As $U_4(p)_{(4,p)}$ is trivial, again from Proposition \ref{zassenhaus}, the commutator $[A,B]$ is uniquely determined by its class in $\mathrm{gr}(U_4(p))$. The claim now follows from Proposition \ref{graded}. 
\end{proof}
\begin{notn} For every $m$ let $(\phi^1_m,\phi^2_m,\ldots,\phi^{m-1}_m)$ denote the coordinates of $\phi_m:U_m(p)\to(\mathbb Z/p)^{m-1}$.
\end{notn}
\begin{lemma}\label{mirror} There is a group homomorphism $\chi:U_3(2)\to U_4(2)$ such that $\phi_4\circ\chi=(\phi^1_3,\phi^2_3,\phi^1_3)$. 
\end{lemma}
\begin{proof} Note that the group $U_3(2)$ is generated by the two elements:
$$x=\begin{pmatrix}
1 & 1 & 0 \\
0 & 1 & 0  \\
0 & 0 & 1 \end{pmatrix}
\quad\textrm{and}\quad
x=\begin{pmatrix}
1 & 0 & 0 \\
0 & 1 & 1  \\
0 & 0 & 1 \end{pmatrix}$$
subject to the relations:
$$x^2=I,\quad y^2=I,\quad [x,y]^2=I,\quad [x,[x,y]]=I,\quad [y,[x,y]]=I,$$
and this system of relations give a presentation of $U_3(2)$. On the other hand the the two elements:
$$a=\begin{pmatrix}
1 & 1 & 0 & 0 \\
0 & 1 & 0 & 0 \\
0 & 0 & 1 & 1 \\
0 & 0 & 0 & 1 \end{pmatrix}
\quad\textrm{and}\quad
b=\begin{pmatrix}
1 & 0 & 0 & 0 \\
0 & 1 & 1 & 0 \\
0 & 0 & 1 & 0 \\
0 & 0 & 0 & 1 \end{pmatrix}$$
of $U_4(2)$ satisfy the same relations:
$$a^2=I,\quad b^2=I,\quad [a,b]^2=I,\quad [a,[a,b]]=I,\quad [b,[a,b]]=I.$$
Therefore there is a group homomorphism $\chi:U_3(p)\to U_4(p)$ such that
$\chi(x)=a$ and $\chi(y)=b$. In particular
$$\phi_4(\chi(x))=(\phi^1_3,\phi^2_3,\phi^1_3)(x)\textrm{ and }
\phi_4(\chi(y))=(\phi^1_3,\phi^2_3,\phi^1_3)(y).$$
Since $x$ and $y$ generate $U_3(p)$, the claim follows. 
\end{proof}

\section{Fibre products and embedding problems}

\begin{defn}\label{unipotents} Let $\overline U_m(p)$ and $\overline{\overline U}_m(p)$ denote the quotient groups:
$$\overline U_m(p)=U_m(p)/Z_m(p)\quad\textrm{and}\quad
\overline{\overline U}_m(p)=U_m(p)/P_m(p),$$
respectively, and let $\omega_{n+1}:U_{n+1}(p)\to\overline U_{n+1}(p)$ and
$\varpi_{n+1}:U_{n+1}(p)\to\overline{\overline U}_{n+1}(p)$ denote the quotient maps. Let $\zeta_{n+1}:\overline U_{n+1}(p)\to(\mathbb Z/p)^{n}$ be the unique homomorphism such that the composition of $\omega_{n+1}$ and $\zeta_{n+1}$ is the homomorphism $\phi_{n+1}$. Similarly let $\kappa_{n+1}:\overline{\overline U}_{n+1}(p)\to(\mathbb Z/p)^{n}$ be the unique group homomorphism such that the composition of $\varpi_{n+1}$ and $\kappa_{n+1}$ is the homomorphism $\phi_{n+1}$. 
\end{defn}
\begin{defn}\label{basic} By a {\it class of embedding problems} $\mathbf B$ we mean a homomorphism $\epsilon:\Gamma\to\Delta$ of finite groups. We say that $\mathbf B$ {\it has abelian kernel} if $\mathrm{Ker}(\epsilon)$ is abelian. We say that $\mathbf B$ is {\it central} if $\mathrm{Ker}(\epsilon)$ is a central subgroup of $\Gamma$. We say that an embedding problem for a group $G$ belongs to a class $\mathbf B$ as above if it is of the form:
$$\xymatrix{
 & G\ar[d]^{\phi}\ar@{.>}[ld]_{\psi}
\\
\Gamma\ar[r]^{\epsilon} &\Delta.}$$
We will let $\mathbf B(\phi)$ denote the latter. Given a homomorphism $\chi:H\to G$ we define the {\it pull-back of $\mathbf B(\phi)$} as the embedding problem $\mathbf B(\phi\circ\chi)$ belonging to the class $\mathbf B$.
\end{defn}
\begin{defn}\label{classes} Let $\mathbf E_n$ denote the class of embedding problems given by the homomorphism $\phi_{n+1}:U_{n+1}(p)\to(\mathbb Z/p)^{n}$. Let $\mathbf D_n$ denote the class of embedding problems given by the homomorphism $\zeta_{n+1}:\overline U_{n+1}(p)\to(\mathbb Z/p)^{n}$.  Let $\mathbf C_n$ denote the class of embedding problems given by the homomorphism $\kappa_{n+1}:\overline{\overline U}_{n+1}(p)\to(\mathbb Z/p)^{n}$. 
\end{defn}
\begin{notn} Let $G,H,J$ be three groups and let $\gamma:G\to J$ and $\chi:H\to J$ be two homomorphisms. The {\it fibre product} $G\times_{\gamma,\chi}H$ is the group:
$$G\times_{\gamma,\chi}H=\{(g,h)\in G\times H\mid\gamma(g)=\chi(h)\}\subseteq
G\times H.$$
For every $n\geq3$ and pro-finite group $G$ let $\mathcal D_n(G)$ denote the set of continuous homomorphisms $\underline a:G\longrightarrow(\mathbb Z/p)^n$ such that the embedding problem $\mathbf D_n(\underline a)$ has a solution. 
\end{notn}
\begin{notn}\label{nono} Now let $G$ be a $p$-group, and let $\alpha=(\alpha_1,\alpha_2,\ldots,\alpha_n)\in\mathcal D_n(G)$ for some $n\geq3$. Let $H$ denote the fibre product $U_{n+1}(p)\times_{\phi_{n+1},\alpha}G$, and let $\rho:H\to G$ be the projection onto the second factor. Let $Z$ be the subgroup
$$Z=\{(a,b)\in U_{n+1}(p)\times_{\phi_{n+1},\alpha}G
\mid a\in Z_{n+1}(p),\ b=1\}$$
of $H$. Finally let $\mathbf B$ be the class of embedding problems $\epsilon:\Gamma\to\Delta$ and let $\mathbf B(\phi)$ be an embedding problem for $G$ belonging to the class $\mathbf B$. 
\end{notn}
\begin{prop}\label{kernel_induction} Assume that the embedding problem
$\mathbf B(\phi\circ\rho)$ has a solution $\psi$ whose restriction onto $Z$ is trivial. Then $\mathbf B(\phi)$ has a solution, too.  
\end{prop}
\begin{proof} Now let $\overline H$ denote the fibre product
$\overline U_{n+1}(p)\times_{\zeta_{n+1},\alpha}G$, and let $\overline{\rho}:
\overline H\to G$ be the projection onto the second factor. Note that the quotient of $H$ by its normal subgroup $Z$ is canonically isomorphic to $\overline H$. Since $\mathrm{Ker}(\psi)$ contains $Z$ by assumption, the homomorphism $\psi$ factors through the quotient map $H\to\overline H$, so there is a solution $\overline{\psi}:\overline H\to\Gamma$ of the embedding problem $\mathbf B(\phi\circ\overline{\rho})$. By assumption there is a solution $\omega:G\to\overline U_{n+1}(p)$ of the embedding problem $\mathbf D_n(\alpha)$. The direct product $\omega\times\mathrm{id}_G:G\to\overline U_{n+1}(p)
\times G$ maps $G$ into $\overline H$. We have a commutative diagram:
$$\xymatrix{ G\ar[r]^{\omega\times\mathrm{id}_G} \ar[rd]^{\mathrm{id}_G} & \overline H\ar[r]^{\overline{\psi}} \ar[d]^{\overline{\rho}} &
\Gamma \ar[d]^{\epsilon} \\
 & G\ar[r]^{\phi} & \Delta.}$$
The composition $\overline{\rho}\circ(\omega\times\mathrm{id}_G)$ is the identity, therefore the composition $\overline{\psi}\circ(\omega\times\mathrm{id}_G)$ is a solution to $\mathbf B(\phi)$.
\end{proof}
\begin{prop}\label{weak_induction} Assume that $n\geq4$, the class $\mathbf B$ has abelian kernel, and the embedding problem $\mathbf B(\phi\circ\rho)$ has a solution for $G$. Then $\mathbf B(\phi)$ has a solution, too.  
\end{prop}
\begin{proof} Let $\psi:H\to\Gamma$ be a solution of $\mathbf B(\phi\circ\rho)$. Then the restriction of $\psi$ onto the subgroup
$$K=\{(a,b)\in U_{n+1}(p)\times_{\phi_{n+1},\alpha}G
\mid a\in K_{n+1}(p),\ b=1\}$$
of $H$ lands in the kernel of
$\mathrm{Ker}(\epsilon)$. The latter is abelian, so $\psi$ is trivial on the subgroup 
$$K'=\{(a,b)\in U_{n+1}(p)\times_{\phi_{n+1},\alpha}G
\mid a\in U_{4,n+1}(p),\ b=1\}$$
by Corollary \ref{derived}. Since $n\geq4$ the group $K'$ contains $Z$ (introduced in Notation \ref{nono}), and hence the claim follows from Proposition \ref{kernel_induction}.  
\end{proof}
\begin{prop}\label{weak_induction3} Assume that $n=3$ and the embedding problem $\mathbf B(\phi\circ\rho)$ has a solution $\psi$. In addition suppose that one of the following conditions is also true:
\begin{enumerate}
\item[$(a)$] the embedding problem $\mathbf E_3(\alpha)$ has a solution,
\item[$(b)$] either $\alpha_1=0$ or $\alpha_3=0$,
\item[$(c)$] we have $p=2$ and $\alpha_1=\alpha_3$.
\end{enumerate}
Then the embedding problem $\mathbf B(\phi)$ has a solution for $H$. 
\end{prop}
\begin{proof} First assume that $(a)$ holds and let $\omega:G\to U_4(p)$ be a solution to $\mathbf E_3(\alpha)$. The direct product $\omega\times\mathrm{id}_G:G\to U_4(p)\times G$ maps $G$ into $H$. The composition $\overline{\rho}\circ(\omega\times\mathrm{id}_G)$ is the identity, therefore the composition
$\sigma\circ(\omega\times\mathrm{id}_G)$ is a solution to $\mathbf B(\phi)$. Next assume that $(b)$ is true. It will be enough to show that $\mathbf E_3(\alpha)$ has a solution by the above. Let $\upsilon_1:U_3(p)\to U_4(p)$ and $\upsilon_3:U_3(p)\to U_4(p)$ be the homomorphisms given by the rules:
$$\begin{pmatrix}
1 & a & b \\
0 & 1 & c  \\
0 & 0 & 1 \end{pmatrix}\mapsto
\begin{pmatrix}
1 & 0 & 0 & 0 \\
0 & 1 & a & b \\
0 & 0 & 1 & c \\
0 & 0 & 0 & 1 \end{pmatrix}
\textrm{ and }
\begin{pmatrix}
1 & a & b \\
0 & 1 & c  \\
0 & 0 & 1 \end{pmatrix}\mapsto
\begin{pmatrix}
1 & a & b & 0 \\
0 & 1 & c & 0 \\
0 & 0 & 1 & 0 \\
0 & 0 & 0 & 1 \end{pmatrix},$$
respectively. First suppose that $\alpha_1=0$. Since $\mathbf D_3(\alpha)$ is solvable there is a homomorphism $\sigma:G\to U_3(p)$ such that $(\alpha_2,\alpha_3)=\zeta_3\circ\sigma$. Then the composition $\sigma\circ\upsilon_1$ is a solution for $\mathbf E_3(\alpha)$. 

The proof in the case when $\alpha_3=0$ is similar. Since $\mathbf D_3(\alpha)$ is solvable there is a homomorphism $\sigma:G\to U_3(p)$ such that $(\alpha_1,\alpha_2)=\zeta_3\circ\sigma$. Then the composition $\sigma\circ\upsilon_3$ is a solution for $\mathbf E_3(\alpha)$. Finally we assume that $(c)$ is true. Since $\mathbf D_3(\alpha)$ is solvable there is a homomorphism
$\sigma:G\to U_3(p)$ such that $(\alpha_1,\alpha_2)=\zeta_3\circ\sigma$. Then the composition $\chi\circ\sigma$, where $\chi:U_3(2)\to U_4(2)$ is the homomorphism in Lemma \ref{mirror}, is a solution for $\mathbf E_3(\alpha)$.
\end{proof}
Now let $\beta:G\to(\mathbb Z/p)^2$ be a homomorphism. 
\begin{prop}\label{cup_induction} Assume that the embedding problem
$\mathbf E_2(\beta\circ\rho)$ has a solution. Then $\mathbf E_2(\beta)$ has a solution, too.
\end{prop}
\begin{proof} Note that the class $\mathbf E_2$ has abelian kernel. Therefore the claim holds when $n\geq4$ by Proposition \ref{weak_induction}. So we may assume that $n=3$ without the loss of generality. We may also suppose that $\alpha_3\neq0$ by Proposition \ref{weak_induction3}. Therefore there is a $g\in G$ such that $\alpha_3(g)=1$. Choose a $u\in U_4(p)$ such that
$\phi_4(u)=\alpha(g)$. Then $(u,g)\in H$ and $(I+E_{13},1)\in H$ (where $1$ is the unit of $G$), while 
$$[(I+E_{13},1),(u,g)]=([I+E_{13},u],[1,g])=(I+E_{14},1)$$
using Lemma \ref{u(3)ids}. Let $\psi:H\to U_3(p)$ be a solution to $\mathbf E_2(\beta\circ\rho)$. Then $\psi((I+E_{13},1))$ lies in $\mathrm{Ker}(\zeta_2)$, which is a central subgroup in $U_3(p)$. Therefore
$$\psi((I+E_{14},1))=\psi([(I+E_{13},1),(u,g)])=[\psi((I+E_{13},1)),
\psi((u,g))]=I.$$
The element $(I+E_{14},1)$ generates the subgroup $Z$ introduced in Notation \ref{nono}, so by Proposition \ref{kernel_induction} the embedding problem $\mathbf E_2(\beta)$ has a solution.
\end{proof}
\begin{notn} For every $1\leq i<j\leq m$ such that $(i,j)\neq(1,m)$ let $\overline{\mathbf b}_{ij}:\overline U_{m}(p)\to U_{j-i+1}(p)$ be the unique homomorphism such that the composition of the quotient map $\omega_m:U_m(p)\to\overline U_m(p)$ and $\overline{\mathbf b}_{ij}$ is the homomorphism $\mathbf b_{ij}$ in Definition \ref{blocks}. For every sequence $c_1,c_2,\ldots,c_m\in H^1(G)=\mathrm{Hom}(G,\mathbb Z/p)$ let $\mathrm{span}( c_1,c_2,\ldots,c_m)\subseteq H^1(G)$ denote the
$\mathbb Z/p$-linear span of these elements. Now let $\gamma:G\to(\mathbb Z/p)^4$ be a homomorphism with coordinates
$\gamma=(\gamma_1,\gamma_2,\gamma_3,\gamma_4)$. 
\end{notn}
\begin{prop}\label{span_red} Assume that $p=2$, the embedding problem $\mathbf D_4(\gamma\circ\rho)$ has a solution $\psi$, but the problem $\mathbf D_4(\gamma)$ does not. Then $n=3$ and one of the following is true:
\begin{enumerate}
\item[$(14)$] we have $\mathrm{Ker}(\overline{\mathbf b}_{14}\circ\psi)\cap Z=\{1\}$ and $\mathrm{span}(\alpha_1,\alpha_3)=
\mathrm{span}(\gamma_1,\gamma_3)$, 
\item[$(25)$] we have $\mathrm{Ker}(\overline{\mathbf b}_{25}\circ\psi)\cap Z=\{1\}$ and $\mathrm{span}(\alpha_1,\alpha_3)=
\mathrm{span}(\gamma_2,\gamma_4)$, 
\end{enumerate}
where $Z\subseteq H$ is the subgroup introduced in Notation \ref{nono}.
\end{prop}
\begin{proof} Note that the class $\mathbf D_4$ has abelian kernel. Therefore $n=3$ by Proposition \ref{weak_induction}. We also know that
$\alpha_1\neq0$ and $\alpha_3\neq0$ by part $(b)$ of Proposition \ref{weak_induction3}.  If $\mathrm{span}(\alpha_1,\alpha_3)$ is one-dimensional then $\alpha_1=\alpha_3$, and hence $\mathbf D_4(\kappa_5\circ\rho)$ has a solution by part $(c)$ of Proposition \ref{weak_induction3}. This is a contradiction, so $\mathrm{span}(\alpha_1,\alpha_3)$ is two-dimensional. 

If $\mathrm{Ker}(\psi)\cap Z\neq\{1\}$, then $\mathrm{Ker}(\psi)\supseteq Z$,  where $Z\subseteq H$ is the subgroup introduced in Notation \ref{nono}, since this subgroup is of order $2$. By Proposition \ref{kernel_induction} the latter is not possible, so $\mathrm{Ker}(\psi)\cap Z=\{1\}$. Since $\overline{\mathbf b}_{14}\times\overline{\mathbf b}_{25}:\overline U_5(2)\to U_4(2)\times U_4(2)$ is injective, we get that either $\mathrm{Ker}(\overline{\mathbf b}_{14}\circ\psi)\cap Z=\{1\}$ or $\mathrm{Ker}(\overline{\mathbf b}_{25}\circ\psi)\cap Z=\{1\}$. Let's consider the first case; the second can be handled similarly. We will show that $(14)$ holds. Since $\mathrm{span}(\alpha_1,\alpha_3)$ is two-dimensional, it will be sufficient to show that $\alpha_1,\alpha_3\in
\mathrm{span}(\gamma_1,\gamma_3)$ in order to conclude the proof.   

First assume to the contrary that $\alpha_1\not\in
\mathrm{span}(\gamma_1,\gamma_3)$. Then there is a $g\in G$ such that
$(\gamma_1,\gamma_2,\gamma_3)(g)=(0,*,0)$ and $\alpha(g)=(1,*,*)$. Choose a $u\in U_4(p)$ such that $\phi_4(u)=\alpha(g)$. Then $(u,g)\in H$ and $(I+E_{24},1)\in H$, while 
$$[(u,g),(I+E_{24},1)]=([u,I+E_{24}],[1,g])=(I+E_{14},1)$$
using Corollary \ref{u(3)ids}. On the other hand
\begin{align*}
\overline{\mathbf b}_{14}\circ\psi((I+E_{14},1))= &  \overline{\mathbf b}_{14}\circ\psi([(u,g),(I+E_{24},1)]) \\
 =  & [\overline{\mathbf b}_{14}\circ\psi((u,g)),
\overline{\mathbf b}_{14}\circ\psi((I+E_{24},1))]=I
\end{align*}
using Corollary \ref{u(3)ids} and $\phi_4\circ\overline{\mathbf b}_{14}\circ\psi((u,g))=(\gamma_1,\gamma_2,\gamma_3)(g)$. Since $(I+E_{24},1)$ generates $Z$, this is a contradiction, so $\alpha_1\in \mathrm{span}(\gamma_1,\gamma_3)$.

Now suppose that $\alpha_3\not\in\mathrm{span}(\gamma_1,\gamma_3)$. Then there is a $g\in G$ such that $\alpha(g)=(*,*,1)$ and
$(\gamma_1,\gamma_2,\gamma_3)(g)=(0,*,0)$. Choose a $u\in U_4(p)$ such that $\phi_4(u)=\alpha(g)$. Then $(u,g)\in H$ and $(I+E_{13},1)\in H$, while 
$$[(I+E_{13},1),(u,g)]=([I+E_{13},u],[1,g])=(I+E_{14},1)$$
using Corollary \ref{u(3)ids} and and $\phi_4\circ\overline{\mathbf b}_{14}\circ\psi((u,g))=(\gamma_1,\gamma_2,\gamma_3)(g)$. On the other hand using the same computation as above we get
$$\overline{\mathbf b}_{14}\circ\psi((I+E_{14},1))=\overline{\mathbf b}_{14}\circ\psi([(I+E_{13},1),(u,g)])=I.$$
This is a contradiction, so $\alpha_3\in\mathrm{span}(\gamma_1,\gamma_3)$.
\end{proof}
\begin{notn}\label{zerodef} let $\theta_{n+1}:\overline U_{n+1}(p)\to\overline U_{n+1}(p)$ denote the quotient map. For every $1\leq i<j\leq m$ such that $j\leq i+3$ let
$\mathbf c_{ij}:\overline{\overline U}_{m}(p)\to\overline U_{j-i+1}(p)$ be the unique homomorphism such that
$\mathbf c_{ij}\circ\varpi_{n+1}=\omega_{j-i+1}\circ\mathbf b_{ij}$. For every $m$ let $(\kappa^1_m,\kappa^2_m,\ldots,\kappa^{m-1}_m)$ denote the coordinates of $\kappa_m:\overline{\overline U}_m(p)\to(\mathbb Z/p)^{m-1}$. Let $\mathbf G_2$ denote the double (or iterated) fibre product:
$$\{(u,g,v)\in U_4(p)\times\overline{\overline U}_5(p)\times U_4(p)\mid\phi_4(u)=(\kappa^1_5,\kappa^2_5,\kappa^3_5)(g),\ 
\phi_4(v)=(\kappa^2_5,\kappa^3_5,\kappa^4_5)(g)\}.$$ 
Let $\rho:\mathbf G_2\to\overline{\overline U}_5(p)$ denote the projection onto the second, middle factor. 
\end{notn}
\begin{prop}\label{zero_case} The embedding problem
$\mathbf D_4(\kappa_5\circ\rho)$ has no solution for $\mathbf G_2$. 
\end{prop}
\begin{proof} Assume to the contrary that $\psi:\mathbf G_2\to\overline U_5(p)$ is a solution of $\mathbf D_4(\kappa_5\circ\rho)$. Let $V\subseteq\mathbf G_2$ be the subgroup:
$$\{(u,g,v)\in U_4(p)\times\overline{\overline U}_5(p)\times U_4(p)\mid\omega_4(u)=\mathbf c_{14}(g),\ 
\omega_4(v)=\mathbf c_{25}(g)\}.$$ 
It is isomorphic to $\overline U_5(p)$, indeed
$$\overline{\mathbf b}_{14}\times\theta_5\times\overline{\mathbf b}_{25}:\overline U_5(p)\to U_4(p)\times\overline{\overline U}_5(p)\times U_4(p)$$
maps $\overline U_5(p)$ isomorphically onto $V$. Note that $\kappa_5\circ\rho|_V$ is $\zeta_5$ under this identification, and hence it is surjective. Since the kernel of $\zeta_5$ is $\overline U_5(p)'$, we get that the composition of $\psi|_V$ and the quotient map $\overline U_5(p)\to\overline  U_5(p)/\overline U_5(p)'$ is surjective. Since $\overline U_5(p)$ is a $p$-group, we get that $\psi|_V$ is surjective. But $V$ and $\overline U_5(p)$ have the same order, as they are isomorphic, therefore $\psi|_V$ is an isomorphism.
\begin{notn} Let $N$ denote the kernel of $\psi$. Let $C=Z_4(p)\times\{1\}\times Z_4(p)\subseteq\mathbf G_2$ and $L=K_4(p)\times\mathrm{Ker}(\kappa_5)\times K_4(p)\subseteq\mathbf G_2$. Let $\sigma_1:\mathbf G_2
\to U_4(p)$ and $\sigma_3:\mathbf G_2\to U_4(p)$ denote the projection onto the first and the third factor, respectively. 
\end{notn}
\begin{lemma}\label{dalshabet} The following hold:
\begin{enumerate}
\item[$(a)$] the order of $N$ is $p^4$.
\item[$(b)$]  the intersection $N\cap C$ is trivial, 
\item[$(c)$]  we have $N\subseteq L$, 
\item[$(d)$] the map $\eta_4\circ\sigma_1\times\eta_4\circ\sigma_3:N\to
(\mathbb Z/p)^2\times(\mathbb Z/p)^2$ is non-trivial. 
\end{enumerate}
\end{lemma}
Note that the homomorphism $\eta_4\circ\sigma_1\times\eta_4\circ\sigma_3|_N$ in part $(d)$ is well-defined because of part $(c)$.
\begin{proof} As $\psi|_V$ is an isomorphism the map $\psi$ is surjective.  The order of $\overline U_5(p)$ is $p^9$, while the order of $\mathbf G_2$ is $p^{13}$, so $(a)$ holds. As $C$ is a subgroup of $V$, and $\psi|_V$ is an isomorphism, part $(b)$ is clear. Since $\zeta_5\circ\psi=\kappa_5\circ\rho$, the subgroup $N$ must lie in the kernel of  $\kappa_5\circ\rho$, which is $L$, so $(c)$ is true. Assume now that
$\eta_4\circ\sigma_1\times\eta_4\circ\sigma_3|_N$ is trivial. The kernel of
$\eta_4\circ\sigma_1\times\eta_4\circ\sigma_3$ in $L$ is the direct sum of $C$ and $L\cap\mathrm{Ker}(\sigma_1\times\sigma_3)$. Since $N\cap C$ is trivial by part $(b)$, the group $N$ injects into $L\cap\mathrm{Ker}(\sigma_1\times\sigma_3)$. However $L\cap\mathrm{Ker}(\sigma_1\times\sigma_3)\cong\mathrm{Ker}(\kappa_5)$ via $\rho$, so its order is $p^3$. But the order of $N$ is $p^4$ by part $(a)$, a contradiction. So $(d)$ holds.
\end{proof}
Assume now that $\eta_4\circ\sigma_1|_N$ is non-trivial; the case when $\eta_4\circ\sigma_3|_N$ is non-trivial can be handled similarly. Let $g=(g_1,g_2,g_3)\in N$ be an element such that $\eta_4\circ\sigma_1(g)$ is non-trivial. Then either the first or the second coordinate of
$\eta_4\circ\sigma_1(g)$ is non-zero. Let's first assume the former. By taking a suitable power of $g$, if this is necessary, we may assume without the loss of generality that $\eta_4\circ\sigma_1(g)=(1,*)$. Note that
$(I+E_{34},\varpi_5(I+E_{34}),I+E_{23})\in\mathbf G_2$ and
\begin{align*}
[(g_1,g_2,g_3),(I+E_{34},\varpi_5(I+E_{34}),I+E_{23})]= & 
 \\
([g_1,I+E_{34}],[g_2,\varpi_5(I+E_{34})],[g_3,I+E_{23}]) =  & (I+E_{14},1,1)\in N
\end{align*}
using Corollary \ref{u(3)ids} and part $(c)$ of Lemma \ref{dalshabet}. Since
$(I+E_{14},1,1)\in C$, this contradicts part $(b)$ of Lemma \ref{dalshabet}. 

Now suppose the latter. We may assume without the loss of generality that
$\eta_4\circ\sigma_1(g)=(*,1)$, as above. Note that
$(I+E_{12},I,I)\in\mathbf G_2$ and
$$[(I+E_{24},I,I),(g_1,g_2,g_3)]=([I+E_{12},g_1],[I,g_2],[I,g_3])
=(I+E_{14},1,1)\in N$$
using Corollary \ref{u(3)ids} and part $(c)$ of Lemma \ref{dalshabet}. Since
$(I+E_{14},1,1)\in C$, this is again a contradiction. 
\end{proof}

\section{Massey envelopes}

\begin{defn} We say that $G$ satisfies {\it weak Massey vanishing for $n$} if for every homomorphism $\underline a:G\to(\mathbb Z/p)^{n}$ such that the embedding problem $\mathbf D_n(\underline a)$ has a solution, the problem $\mathbf E_n(\underline a)$ also has a solution. We say that $G$ satisfies {\it strong Massey vanishing for $n$} if for every homomorphism $\underline a:G\to(\mathbb Z/p)^{n}$ such that the embedding problem
$\mathbf C_n(\underline a)$ has a solution, the problem $\mathbf E_n(\underline a)$ also has a solution. 
\end{defn}
Now we will concentrate on the case $p=2$. The main result of this section is:
\begin{thm}\label{counter} There is a pro-$2$ group $G$ which satisfies weak Massey vanishing for $n\geq3$, but does not satisfy strong Massey vanishing for $n=4$.
\end{thm}
The proof will occupy the rest of this section. 
\begin{defn} Let $\mathcal D_{\leq n}(G)$ denote the union
$\bigcup_{3\leq k\leq n}\mathcal D_k(G)$. When $G$ is finite, the set $\mathcal D_{\leq n}(G)$ is also finite. Every homomorphism $\alpha:G\to H$ of pro-finite groups induces a map $\alpha^*:\mathcal D_{\leq n}(H)\to\mathcal D_{\leq n}(G)$ via composition with $\alpha$ which is injective when $\alpha$ is surjective. We will identify $\mathcal D_{\leq n}(H)$ with its image under $\alpha^*$ in this case. Finally for every $\beta\in\mathcal D_{\leq n}(G)$ let $d(\beta)$ denote the unique integer such that $\beta\in\mathcal D_{d(\beta)}(G)$. For every set $S$ let $|S|$ denote its cardinality, that is, the minimal ordinal in bijection with $S$. 
\end{defn}
\begin{lemma}\label{bound} For every non-trivial finite $p$-group $G$ and $n\geq3$ we have
$|\mathcal D_{\leq n}(G)|\geq n$. 
\end{lemma}
\begin{proof} It will be sufficient to show that $|\mathcal D_n(G)|\geq 3$ for every $n\geq3$, as $3(n-2)\geq n$ when $n\geq3$. Since $G$ is a finite $p$-group there is a non-zero homomorphism $\mu:G\to\mathbb Z/p$. Now let $\underline m_i:G\to(\mathbb Z/p)^n$ be the homomorphism whose $i$-th coordinate is $\mu $ and all other coordinate is the zero map for every $i=1,2,\ldots,n$. These maps are pairwise different, so $|\mathcal D_n(G)|\geq n\geq3$. 
\end{proof}
\begin{defn} Now let $G$ be a non-trivial $p$-group. We can construct three sequences of objects of the following kind:
\begin{enumerate}
\item[$(a)$] a finite group $G_k$ for every $k=0,1,\ldots$, 
\item[$(b)$] a surjective homomorphism $\pi_k:G_k\to G_{k-1}$ for every $k=1,2,\ldots$, 
\item[$(c)$] a bijection $\iota_k:\mathcal D_{\leq k+3}(G_k)\to|\mathcal D_{\leq k+3}(G_k)|$ for every $k=0,1,\ldots$,
\end{enumerate}
with the following properties:
\begin{enumerate}
\item[$(i)$] we have $G_0=G$,
\item[$(ii)$] we have $G_k=U_{d(\alpha)+1}(p)\times_{\phi_{d(\alpha)+1},\alpha}G_{k-1}$, where $\alpha\in\mathcal D_{\leq k+2}(G_{k-1})$ is the pre-image of $k-1$ with respect to $\iota_{k-1}$ for $k\geq 1$,
\item[$(iii)$] the map $\pi_k:G_k\to G_{k-1}$ is the projection onto the second factor of the fibre product $G_k$ for $k\geq 1$, 
\item[$(iv)$] the restriction of $\iota_k$ onto $\mathcal D_{\leq k+2}(G_{k-1})\subset\mathcal D_{\leq k+3}(G_k)$ (where the inclusion is with respect to $\pi_k$) is $\iota_{k-1}$ for $k\geq 1$. 
\end{enumerate}
In the construction we only have some freedom in the choice of $\iota_k$. Note that we can perform the construction in $(ii)$ as $|\mathcal D_{\leq k+2}(G_{k-1})|\geq k+2$ by Lemma \ref{bound}, so $k-1$ is in the image of $\iota_{k-1}$. Also note that $\pi_k$ is surjective since $\phi_n$ is, for every $n$. Given a sequence above, let $\mathcal M(G)$ denote the projective limit of the system:
$$\xymatrix{ \cdots\ar[r]^{\pi_{k+1}} & G_k \ar[r]^{\pi_k} &
G_{k-1}\ar[r]^{\pi_{k-1}} & \cdots}$$
by slight abuse of notation. We will call $\mathcal M(G)$ {\it a Massey envelope of $G$}. It is equipped with a surjective homomorphism $\pi^k:\mathcal M(G)\to G_k$ for every $k=0,1,\ldots$. We will let $\pi$ denote this map when $k=0$. 
\end{defn}
\begin{rem}\label{ponpon} If $G$ is a $p$-group, then it is easy to prove using induction that $G_k$ is a $p$-group, too. Indeed $p$-groups are closed under direct products and taking subgroups, so under fibre products, too. As a consequence we get that $\mathcal M(G)$ is a pro-$p$ group in this case. 
\end{rem}
\begin{lemma}\label{envelope} The Massey envelope $\mathcal M(G)$ satisfies weak Massey vanishing for every $n\geq3$. 
\end{lemma}
\begin{proof} Let $\alpha\in\mathcal D_n(\mathcal M(G))$ for some $n\geq3$. Then there is an index $k$ such that $\alpha$ is already an element of
$(\pi^k)^*(\mathcal D_n(G_k))\subset\mathcal D_n(\mathcal M(G))$. By Lemma \ref{bound} we may assume that $\alpha$ is the pre-image of $k$ with respect to $\iota_k$ without the loss of generality by enlarging $k$, if this is necessary. Therefore $G_{k+1}$ is the fibre product $U_{d(\alpha)+1}(p)\times_{\phi_{d(\alpha)+1},\alpha}G_k$. Clearly $\mathbf E_{d(\alpha)}(\alpha)$ is solvable over $G_{k+1}$, the solution being the projection of this fibre product onto its first factor. Therefore the pull-back of this embedding problem is solvable over $\mathcal M(G)$, too. 
\end{proof}
\begin{lemma}\label{universal} Let $H$ be a pro-finite group which satisfies weak Massey vanishing for every $n\geq3$ and let $\chi:H\to G$ be a homomorphism, where $G$ is a non-zero $p$-group. Then there is a homomorphism
$\widetilde\chi:H\to\mathcal M(G)$ such that $\chi=\pi\circ\widetilde\chi$. 
\end{lemma}
\begin{proof} We are going to construct a sequence of homomorphisms
$\chi_k:H\to G_k$ by induction on $k$ such that
\begin{enumerate}
\item[$(i)$] we have $\chi_0=\chi$, 
\item[$(ii)$] we have $\pi_k\circ\chi_k=\chi_{k-1}$ for every $k=1,2,\ldots$. 
\end{enumerate}
The limit $\widetilde\chi$ of the homomorphisms $\chi_k$ will have the required properties. Assume now that $\chi_{k-1}$ has been constructed already. Let
$\alpha\in\mathcal D_{\leq k+2}(G_{k-1})$ be the pre-image of $k-1$ with respect to $\iota_{k-1}$, as above. Then $\alpha\circ\chi_{k-1}\in\mathcal D_{\leq k+2}(H)$, and hence the embedding problem $\mathbf E_{d(\alpha)}(\alpha\circ\chi_{k-1})$ has a solution $\sigma:H\to U_{d(\alpha)+1}(p)$ by our assumptions. The direct product $\sigma\times\chi_{k-1}:H\to U_{d(\alpha)+1}(p)\times G_{k-1}$ lands in the fibre product $G_k=U_{d(\alpha)+1}(p)\times_{\phi_{d(\alpha)+1},\alpha}G_{k-1}$, since $\sigma$ is a solution of $\mathbf E_{d(\alpha)}(\alpha\circ\chi_{k-1})$, and so it furnishes a homomorphism $\chi_k:H\to G_k$. By construction the composition of $\chi_k$ and the projection $\pi_k$ of $G_k$ onto its second factor is $\chi_{k-1}$. 
\end{proof}
\begin{lemma}\label{moveit} Assume that  the embedding problem $\mathbf D_4(\kappa_5\circ\pi)$ has a solution, and let $H$ be a pro-finite group which satisfies weak Massey vanishing for every $n\geq3$ equipped with a homomorphism $\phi:H\to \overline{\overline U}_5(p)$. Then the embedding problem $\mathbf D_4(\kappa_5\circ\phi)$ is solvable.
\end{lemma}
\begin{proof} By Lemma \ref{universal} there is a homomorphism
$\widetilde\phi:H\to\mathcal M(G)$ such that $\phi=\pi\circ\widetilde\phi$. According to our assumptions we also have a solution $\sigma:\mathcal M(G)
\to U_5(p)$ to $\mathbf D_4(\kappa_5\circ\pi)$. Then $\sigma\circ\widetilde\phi$ is a solution of $\mathbf D_4(\kappa_5\circ\phi)=\mathbf D_4(\kappa_5\circ\pi\circ
\widetilde\phi)$. 
\end{proof}
\begin{thm}\label{manilla} The embedding problem $\mathbf D_4(\kappa_5\circ\pi)$ has no solution for the Massey envelope $\mathcal M(\overline{\overline U}_5(2))$. 
\end{thm}
By Lemma \ref{envelope} this result implies Theorem \ref{counter}, since
$\mathcal M(\overline{\overline U}_5(2))$ is a $2$-group, as we already noticed in Remark \ref{ponpon}. We will need some lemmas.
\begin{lemma}\label{cup_rels5} The embedding problems $\mathbf E_2((\kappa^1_5,\kappa^3_5))$ and $\mathbf E_2((\kappa^2_5,\kappa^4_5))$ have no solutions for $\overline{\overline U}_5(p)$. 
\end{lemma}
\begin{proof} It will be enough to prove that the embedding problems $\mathbf E_2((\phi^1_5,\phi^3_5))$ and $\mathbf E_2((\phi^2_5,\phi^4_5))$ have no solutions for $U_5(p)$. By Dwyer's theorem it will be sufficient to show that the cup products $\phi^1_5\cup\phi^3_5$ and $\phi^2_5\cup\phi^4_5$ are non-zero. By Lemma \ref{commutator} the elements $I+E_{12}$ and $I+E_{34}$ commute, and are also $p$-torsion, so the subgroup $A$ they generate is isomorphic to $(\mathbb Z/p)^2$.  The pull-back of $\phi^1_5\cup\phi^3_5$ onto $A$ is non-zero by the K\"unneth formula for cohomology with coefficients in $\mathbb Z/p$. Therefore $\phi^1_5\cup\phi^3_5$  is non-zero, too. We can argue similarly for $\phi^2_5\cup\phi^4_5$ by pulling it back to the subgroup generated by $I+E_{23}$ and $I+E_{45}$.
\end{proof}
In the next two lemmas $G$ is an arbitrary group. 
\begin{lemma}\label{tripleset} Let
$\alpha_1,\alpha_2,\alpha_3\in H^1(G)$ and $\gamma_1,\gamma_2,\gamma_3\in H^1(G)$  be such that
$$\mathrm{span}(\alpha_1,\alpha_3)=\mathrm{span}(\gamma_1,\gamma_3) \textrm{ and }
\langle\alpha_1,\alpha_2,\alpha_3\rangle\cap
\langle\gamma_1,\gamma_2,\gamma_3\rangle\neq\emptyset.$$
Then $\langle\alpha_1,\alpha_2,\alpha_3\rangle=
\langle\gamma_1,\gamma_2,\gamma_3\rangle$. 
\end{lemma}
\begin{proof} Recall that for every $\beta_1,\beta_2,\beta_3\in H^1(G)$ the Massey product set $\langle\beta_1,\beta_2,\beta_3\rangle$, if it is non-empty, is a coset of the subgroup $\beta_1\cup H^1(G)+H^1(G)\cup\beta_3\subseteq H^2(G)$. However
$$\alpha_1\cup H^1(G)+H^1(G)\cup\alpha_3=\gamma_1\cup H^1(G)+H^1(G)\cup\gamma_3,$$
since $\mathrm{span}(\alpha_1,\alpha_3)=\mathrm{span}(\gamma_1,\gamma_3)$ and the cup product is bilinear and alternating. We get that both
$\langle\alpha_1,\alpha_2,\alpha_3\rangle$ and
$\langle\gamma_1,\gamma_2,\gamma_3\rangle$, being non-empty, are cosets of the same subgroup, and as their intersection is non-empty, they are equal. 
\end{proof}
\begin{notn} Let $\mathbf F_n$ denote the class of embedding problems given by the homomorphism $\omega_{n+1}:U_{n+1}(p)\to\overline U_{n+1}(p)$. Since $U_{n+1}(p)$ is a central extension of $\overline U_{n+1}(p)$, for every group homomorphism $\phi:G\to\overline U_{n+1}(p)$ the embedding problem $\mathbf F_n(\phi)$ is central. Since $\mathrm{Ker}(\omega_{n+1})=Z_{n+1}(p)\cong\mathbb Z/p$, the obstruction class $o(\mathbf F_n(\phi))$ lies in $H^2(G)$. Let $(\alpha_1,\alpha_2,\ldots,\alpha_n):G\to(\mathbb Z/p)^n$ be an arbitrary homomorphism.
\end{notn}
\begin{lemma}\label{little_d} We have
$$\langle\alpha_1,\alpha_2,\ldots,\alpha_n\rangle=\{o(\mathbf F_n(\phi))\mid
\textrm{$\phi$ is a solution of $\mathbf D_n((\alpha_1,\alpha_2,\ldots,\alpha_n))$}\}.$$
\end{lemma}
\begin{proof} Recall that for every solution $\phi$ of $\mathbf D_n((\alpha_1,\alpha_2,\ldots,\alpha_n))$ the obstruction class $o(\mathbf F_n(\phi))$ is the $n$-fold Massey product with respect to the defining system corresponding to $\phi$ in Dwyer's theorem. Therefore the lemma is just a convenient reformulation of the latter. 
\end{proof}
\begin{proof}[Proof of Theorem \ref{manilla}] Consider the projective system:
$$\xymatrix{ \cdots\ar[r]^{\pi_{k+1}} & G_k \ar[r]^{\pi_k} &
G_{k-1}\ar[r]^{\pi_{k-1}} & \cdots}$$
constructed in Definition 0.9 for $G_0=\overline{\overline U}_5(2)$. Set $\rho_0$ be the identity map of $G_0$ and for every $k\geq1$ let $\rho_k:G_k\to G_0$ denote the composition:
$$\pi_1\circ\cdots\circ\pi_{k-1}\circ\pi_k.$$
We are going to show by induction on $k$ that $\mathbf E_2((\kappa^1_5,\kappa^3_5)\circ\rho_k)$, $\mathbf E_2((\kappa^1_5,\kappa^3_5)\circ\rho_k)$ and $\mathbf D_4(\kappa_5\circ\rho_k)$ have no solutions for $G_k$. Since every group homomorphism $\mathcal M(\overline{\overline U}_5(2))\to\overline U_5(2)$ factors through $\pi^k$ for some $k$, this implies the theorem. With all our preparations it is easy to prove that $\mathbf E_2((\kappa^1_5,\kappa^3_5)\circ\rho)$ and $\mathbf E_2((\kappa^1_5,\kappa^3_5)\circ\rho)$ have no solutions for $G_k$. Indeed the $k=0$ case is just Lemma \ref{cup_rels5}, while the induction step follows at once from Proposition \ref{cup_induction}.

Next we prove that $\mathbf D_4(\kappa_5\circ\rho_k)$ has no solutions for $G_k$. Note that Lemmas \ref{envelope} and \ref{moveit} together imply if
$\mathbf D_4(\kappa_5\circ\rho_k)$ has no solutions for a particular Massey envelope, then it does not have solutions for all such envelopes. Therefore we may assume without the loss of generality that
$\iota_0((\kappa^1_5,\kappa^2_5,\kappa^3_5))=0$ and
$\iota_0((\kappa^2_5,\kappa^3_5,\kappa^4_5))=1$. In this case $G_2$ is the group $\mathbf G_2$ introduced in Notation \ref{zerodef}. Therefore $\mathbf D_4(\kappa_5\circ\rho_2)$ has no solutions for $G_2$ by Proposition \ref{zero_case}. (Since $G_0$ and $G_1$ are quotients of $G_2$, we also get that $\mathbf D_4(\kappa_5\circ\rho_0)$ and $\mathbf D_4(\kappa_5\circ\rho_1)$ have no solutions, either.)

Now assume that $\mathbf D_4(\kappa_5\circ\rho_k)$ has no solutions for some $k\geq2$ and let's prove that $\mathbf D_4(\kappa_5\circ\rho_{k+1})$ has no solutions, either. We will prove the claim indirectly, so let's suppose that $\mathbf D_4(\kappa_5\circ\rho_{k+1})$ has a solution $\psi$. Write $\alpha=(\alpha_1,\alpha_2,\ldots)$ for the pre-image of $k$ with respect to
$\iota_k$. By Proposition \ref{span_red} we have $n=3$. The key fact we need to show is the following
\begin{prop}\label{ginger} The Massey product $\langle\alpha_1,\alpha_2,\alpha_3\rangle$ contains zero. 
\end{prop}
Indeed, the proof of Therem \ref{manilla} is now easy; by Proposition \ref{ginger} and Dwyer's theorem the embedding problem $\mathbf E_3(\alpha)$ has a solution. Therefore $\mathbf D_4(\kappa_5\circ\rho_k)$ has a solution by part $(a)$ of Proposition \ref{weak_induction3}. But this is a contradiction. It remains to show Proposition \ref{ginger}, which we will do in several steps. By Proposition \ref{span_red} either $\mathrm{span}(\alpha_1,\alpha_3)=\mathrm{span}(\kappa_5^1,\kappa_5^3)$ and $\mathrm{Ker}(\overline{\mathbf b}_{14}\circ\psi)\cap Z=\{1\}$, or we have $\mathrm{span}(\alpha_1,\alpha_3)=
\mathrm{span}(\kappa_5^2,\kappa_5^4)$ and $\mathrm{Ker}(\overline{\mathbf b}_{25}\circ\psi)\cap Z=\{1\}$, where $Z\subset G_{k+1}$ is the subgroup  
$$Z=\{(a,b)\in U_4(2)\times_{\phi_4,\alpha}G_k
\mid a\in Z_4(2),\ b=1\}.$$
Let us consider the first case; the second can be handled similarly. 
\begin{lemma}\label{centralo} The homomorphism $\overline{\mathbf b}_{14}\circ\psi$ maps $Z$ into $Z_4(2)\subset U_4(2)$. 
\end{lemma}
\begin{proof} Note that $\rho_k$ is surjective, since it is a composition of surjective maps. Since $\phi_4\circ\overline{\mathbf b}_{14}\circ\psi$ is $(\kappa^1_5,\kappa^2_5,\kappa^3_5)\circ\rho_k$, we get that it is surjective. Since the kernel of $\phi_4$ is $U_4(2)'$, we can conclude that the composition of $\overline{\mathbf b}_{14}\circ\psi$ and the quotient map $U_4(2)\to U_4(2)'$ is surjective. Since $U_4(2)$ is a $2$-group, the latter implies that $\overline{\mathbf b}_{14}\circ\psi$ is surjective. Therefore it maps the centre of $G_{k+1}$ into the centre of $U_4(2)$, which is $Z_4(2)$. Since $Z$ lies in the centre of $G_{k+1}$, the claim is now clear. 
\end{proof}
\begin{defn} Let $\overline K_m\subseteq\overline U_m(p)$ be the image of $K_m(p)$ under the quotient map $U_m(p)\to\overline U_m(p)$. Let
$\overline{\eta}_m:\overline K_m(p)\to(\mathbb Z/p)^{m-2}$ be the unique homomorphism such that the composition of the quotient map $K_m(p)\to\overline K_m(p)$ and $\overline{\eta}_m$ is $\eta_m$. 
\end{defn}
Let $\overline G_{k+1}$ denote the fibre product $\overline U_4(2)\times_{\zeta_4,\alpha}G_k$. Note that the quotient of $G_{k+1}$ by its normal subgroup $Z$ is canonically isomorphic to $\overline G_{k+1}$. Therefore by Lemma \ref{centralo} there is a unique homomorphism $\overline{\psi}:\overline G_{k+1}\to\overline U_4(2)$ such that the composition of the quotient map $G_{k+1}\to\overline G_{k+1}$ and $\overline{\psi}$ is $\overline{\mathbf b}_{14}\circ\psi$. Let $K\subseteq G_{k+1}$ denote the subgroup
$$K=\{(a,b)\in U_4(2)\times_{\phi_4,\alpha}G_k
\mid a\in K_4(2),\ b=1\},$$
and let $\overline K\subseteq\overline G_{k+1}$ be its image under the quotient map $G_{k+1}\to\overline G_{k+1}$. 
\begin{lemma}\label{hunters} The map $\overline{\eta}_4\circ\overline{\psi}:\overline K\to(\mathbb Z/2)^2$ is an isomorphism. 
\end{lemma}
\begin{proof} Assume that the claim is false. Since $\overline K\cong(\mathbb Z/2)^2$ this means that the kernel of $\overline{\eta}_4\circ\overline{\psi}|_K$ is non-trivial. Let $(g,1)\in K$ be a lift of a non-zero element $(\overline g,1)\in\mathrm{Ker}(\overline{\eta}_4\circ\overline{\psi}|_K)$ with respect to the quotient map $K\to\overline K$. Then $\eta_4\circ\sigma((g,1))\neq0$, where
$\sigma:G_{k+1}\to U_4(2)$ is the projection onto the first factor, since $\overline g$ is non-zero. Then either the first or the second coordinate of $\eta_4\circ\sigma(g)$ is non-zero.

Let's first assume the former. Since $p=2$ we have $\eta_4\circ\sigma((g,1))=(1,*)$. Let $h\in G_k$ be such that $\alpha(h)=(0,*,1)$. This is possible since $\alpha_1$ and $\alpha_3$ are linearly independent. Choose a $u\in U_4(2)$ such that $\phi_4(u)=\alpha(h)$. Then $(u,h)\in G_{k+1}$ and 
$$[(g,1),(u,h)]=([g,u],[1,h])=(I+E_{14},1)$$
using Corollary \ref{u(3)ids}. On the other hand
\begin{align*}
\overline{\mathbf b}_{14}\circ\psi((I+E_{14},1))= &  \overline{\mathbf b}_{14}\circ\psi([(u,g),(I+E_{24},1)]) \\
 =  & [\overline{\mathbf b}_{14}\circ\psi((u,g)),
\overline{\mathbf b}_{14}\circ\psi((I+E_{24},1))]=I
\end{align*}
using Corollary \ref{u(3)ids}, since $\eta_4\circ\overline{\mathbf b}_{14}\circ\psi((g,1))=\overline{\eta}_4\circ\overline{\psi}(\overline g,1)=(0,0)$. This is not possible as $I+E_{14}$ generates $Z$.

Now suppose the latter. Since $p=2$ we have $\eta_4\circ\sigma((g,1))=
(*,1)$. Let $h\in G_k$ be such that $\alpha(h)=(1,*,0)$. This is possible since
$\alpha_1$ and $\alpha_3$ are linearly independent. Choose a $u\in U_4(2)$ such that $\phi_4(u)=\alpha(g)$. Then $(u,h)\in G_{k+1}$ and 
$$[(u,h),(g,1)]=([u,h],[h,1])=(I+E_{14},1)$$
using Corollary \ref{u(3)ids}. On the other hand using the same computation as above we get
$$\overline{\mathbf b}_{14}\circ\psi((I+E_{14},1))=\overline{\mathbf b}_{14}\circ\psi([(u,h),(g,1)])=I.$$
This is a contradiction.
\end{proof}
By assumption $\mathbf D_3(\alpha)$ has a solution $\beta:G_k\to\overline U_4(2)$. Then the direct product $\beta\times\mathrm{id}_{G_k}:G_k\to\overline U_4(2)\times G_k$ maps $G_k$ into $\overline G_{k+1}$. For the sake of simple notation let $\gamma\star\beta$ denote the composition $\gamma\circ(\beta\times\mathrm{id}_{G_k})$ for every homomorphism $\gamma:\overline G_{k+1}\to H$, where $H$ is any group. 
\begin{lemma}\label{moronic} There is a choice of $\beta$ such that the obstruction class $o(\mathbf F_4(\overline{\psi}\star\beta))$ is non-zero. 
\end{lemma}
\begin{proof} Let $\sigma_1:G_2=\mathbf G_2\to U_4(2)$ be the projection onto the first factor, as in the proof of Proposition \ref{zero_case}. Then $\sigma_1\circ\pi_2\circ\cdots\circ\pi_{k-1}\circ\pi_k$ is a solution to the embedding problem $\mathbf E_3((\kappa^1_5,\kappa^2_5,\kappa^3_5)\circ\rho_k)$. Therefore $\langle\kappa_5^1\circ\rho_k,\kappa_5^2\circ\rho_k,\kappa_5^3\circ\rho_k\rangle$ is the set
$\rho^*_k(\kappa^1_5)\cup H^1(G_k)+H^1(G_k)\cup\rho^*_k(\kappa^3_5)$. By Dwyer's theorem
$\rho^*_k(\kappa^1_5)\cup\rho^*_k(\kappa^3_5)$ is non-zero, since
$\mathbf E_2((\kappa^1_5,\kappa^3_5)\circ\rho_k)$ has no solutions for $G_k$. So $\langle\kappa_5^1\circ\rho_k,\kappa_5^2\circ\rho_k,\kappa_5^3\circ\rho_k\rangle$ contains a non-zero element. Therefore it will be sufficient to show that every solution of $\mathbf D_3((\kappa^1_5,\kappa^2_5,\kappa^3_5)\circ\rho_k)$ can be written in the form $\overline{\psi}\star\beta$ for some choice of $\beta$ by Lemma \ref{little_d}.

Let $\mathbf B$ be a central class of embedding problems given by $\epsilon:\Gamma\to\Delta$. If $\gamma$ is a solution of $\mathbf B(\lambda)$ for some group homomorphism $\lambda:G\to\Delta$, then every solution of $\mathbf B(\lambda)$ is of the form $\lambda\cdot\delta$ for a unique group homomorphism $\delta:G\to\mathrm{Ker}(\epsilon)$, and conversely every such product is a solution of $\mathbf B(\lambda)$. Now fix a solution $\beta$ of $\mathbf D_3(\alpha)$. Since $\mathbf D_3$ is a central class of embedding problems, the solutions of $\mathbf D_3(\alpha)$ are of the form $\beta\cdot\delta$, where $\delta:G_k\to\overline K_4(2)$ is an arbitrary homomorphism, by the above. Let $\mu:\overline K_4(2)\to\overline K$ be the isomorphism given by the rule $g\mapsto(g,1)$. Then
$$\overline{\psi}\star(\beta\cdot\delta)=(\overline{\psi}\star\beta)\cdot
(\overline{\psi}\circ\mu\circ\delta)$$
for every homomorphism $\delta:G_k\to\overline K_4(2)$. The claim now follows from Lemma \ref{hunters} and the fact that $\overline{\eta}_4$ is an isomorphism.
\end{proof}
Now we can conclude the proof of Proposition \ref{ginger}. Fix a choice of
$\beta$ such that $o(\mathbf F_4(\overline{\psi}\star\beta))\neq0$ and let
$\overline{\sigma}:\overline G_{k+1}\to\overline U_4(2)$ be the projection onto the first factor. If $o(\mathbf F_4(\overline{\sigma}\star\beta))$ is zero, then $\langle\alpha_1,\alpha_2,\alpha_3\rangle$ contains zero by Lemma \ref{little_d}. Therefore we may assume that $o(\mathbf F_4(\overline{\sigma}\star\beta))\neq0$ without the loss of generality. Let $\widetilde G_k$ denote the pre-image of $\beta\times\mathrm{id}_{G_k}(G_k)$ with respect to the quotient map $G_{k+1}\to\overline G_{k+1}$. Then  the kernel of the induced projection $\tau:\widetilde G_k\to G_k$ is $Z$. Since $Z$ is a central subgroup in $G_{k+1}$, the natural outer action of $G_k$ on $Z$ induced by conjugation is trivial. Therefore the natural $G_k$-action on $H^*(Z,\mathbb Z/2)$ is trivial, too. In particular the inflation-reflection exact sequence for the trivial module
$\mathbb Z/2$ over the pair $Z\triangleleft\widetilde G_k$ is:
$$\xymatrix{ \cdots\ar[r] & H^1(Z,\mathbb Z/2)\ar[r] & H^2(G_k,\mathbb Z/2)\ar[r]^{\tau^*} &
H^2(\widetilde G_k,\mathbb Z/2).}$$
Since $Z\cong\mathbb Z/2$, we get that $H^1(Z,\mathbb Z/2)=\mathrm{Hom}(Z,\mathbb Z/2)\cong\mathbb Z/2$, and hence the kernel of $\tau^*:H^2(G_k,\mathbb Z/2)\to H^2(\widetilde G_k,\mathbb Z/2)$ is at most one-dimensional. 

Both $\mathbf F_4((\overline{\psi}\star\beta)\circ\tau)$ and $\mathbf F_4((\overline{\sigma}\star\beta)\circ\tau)$ have a solution for $\widetilde G_k$, namely $\overline{\mathbf b}_{14}\circ\psi|_{\widetilde G_k}$ and
$\sigma|_{\widetilde G_k}$, respectively, where $\sigma:G_{k+1}\to U_4(2)$ is the projection onto the first factor. Therefore by the naturality of obstruction classes both $o(\mathbf F_4(\overline{\psi}\star\beta))$ and $o(\mathbf F_4(\overline{\sigma}\star\beta))$ lie in the kernel of $\tau^*:H^2(G_k,\mathbb Z/2)\to H^2(\widetilde G_k,\mathbb Z/2)$. Since both $o(\mathbf F_4(\overline{\psi}\star\beta))$ and $o(\mathbf F_4(\overline{\sigma}\star\beta))$ are non-zero, and $\mathrm{Ker}(\tau^*)$ is at most one-dimensional, we get that $o(\mathbf F_4(\overline{\psi}\star\beta))=o(\mathbf F_4(\overline{\sigma}\star\beta))$. Therefore by Lemma \ref{tripleset} we get that $\langle\alpha_1,\alpha_2,\alpha_3\rangle=
\langle\kappa_5^1\circ\rho_k,\kappa_5^2\circ\rho_k,\kappa_5^3\circ\rho_k\rangle$. Since the latter contains $0$, as we saw in the proof of Lemma \ref{moronic}, we get that the former contains $0$, too.
\end{proof}

\end{document}